\documentclass[11pt]{article}
\usepackage[letterpaper,left=1.1in,right=1.1in,top=1.1in,bottom=1.1in]{geometry}
\usepackage{amsmath,amsthm}

\theoremstyle{plain}
\newtheorem{theorem}{Theorem}
\newtheorem{lemma}[theorem]{Lemma}
\newtheorem{corollary}[theorem]{Corollary}

\theoremstyle{definition}
\newtheorem{definition}[theorem]{Definition}
\newcommand{\newsiamthm}[2]{\theoremstyle{definition}\newtheorem{#1}[theorem]{#2}}
\newcommand{\newsiamremark}[2]{\theoremstyle{remark}\newtheorem{#1}[theorem]{#2}}
\newcommand{\headers}[2]{}
\newcommand{\email}[1]{\texttt{#1}}
\newenvironment{keywords}%
  {\par\medskip\noindent\textbf{Key words. }\ignorespaces}{\par\medskip}
\newenvironment{AMS}%
  {\par\noindent\textbf{AMS subject classifications. }\ignorespaces}{\par\medskip}

\usepackage{amssymb}
\usepackage{bm}
\usepackage{graphicx}
\usepackage{booktabs}
\usepackage{subcaption}
\usepackage{xcolor}
\usepackage[normalem]{ulem}
\usepackage{url}

\definecolor{peiblue}{rgb}{0.00,0.25,0.75}
\definecolor{peigray}{rgb}{0.45,0.45,0.45}

\long\def\cjout#1{}
\newcommand{\TheTitle}{Freezing Calder\'on: Robust Preconditioning for Maxwell Scattering under Shape Uncertainty}

\title{\TheTitle}

\author{
Paul Escapil-Inchausp\'e\thanks{%
  Facultad de Ingenier\'ia y Ciencias, Universidad Adolfo Ib\'a\~nez, Santiago,
  Chile (\email{paul.escapil@edu.uai.cl}).}
\and
Carlos Jerez-Hanckes\thanks{%
  Department of Mathematics and Digital Futures, KTH Royal Institute of Technology,
    Lindstedtsv\"agen 25, Stockholm, Sweden; and Inria Chile, Apoquindo 2827, Las Condes,
    Santiago, Chile (\email{carlosjh@kth.se}).
    Corresponding author}
}

\headers{Robust Preconditioning for Maxwell Scattering}
{P.~Escapil-Inchausp\'e, C.~Jerez-Hanckes}

\newcommand{\bs}[1]{\boldsymbol{#1}}
\newcommand{\R}{\mathbb{R}}
\newcommand{\C}{\mathbb{C}}
\newcommand{\N}{\mathbb{N}}
\newcommand{\calT}{\mathcal{T}}
\newcommand{\calK}{\mathcal{K}}
\newcommand{\calP}{\mathcal{P}}
\newcommand{\calL}{\mathcal{L}}
\DeclareMathOperator{\dive}{div}
\DeclareMathOperator{\curl}{curl}
\newsiamremark{remark}{Remark}
\newsiamthm{assumption}{Assumption}
\makeatletter
\newenvironment{sgassumption}[2]{%
  \def\thetheorem{#1}\def\theassumption{#1}%
  \assumption[#2]}{\endassumption}
\makeatother

\newcommand{\Gnom}{\Gamma_0}
\newcommand{\Gpert}{\Gamma_{\bs y}}

\usepackage{tikz}
\usetikzlibrary{calc,arrows.meta}
\definecolor{nomblue}{HTML}{1565C0}
\definecolor{pergreen}{HTML}{2E7D32}

\newcommand{\meshpatch}[5]{%
  \coordinate (g1) at ($ (#2)!2/3!($ (#3)!0.5!(#4) $) $);
  \coordinate (g2) at ($ (#2)!2/3!($ (#3)!0.5!(#5) $) $);
  \coordinate (m12) at ($ (#2)!0.5!(#3) $);
  \coordinate (m13) at ($ (#2)!0.5!(#4) $);
  \coordinate (m23) at ($ (#3)!0.5!(#4) $);
  \coordinate (m14) at ($ (#2)!0.5!(#5) $);
  \coordinate (m24) at ($ (#3)!0.5!(#5) $);
  \fill[#1!12]
    (#2) -- (m13) -- (g1) -- (m12) -- cycle
    (#3) -- (m12) -- (g1) -- (m23) -- cycle
    (#2) -- (m14) -- (g2) -- (m12) -- cycle
    (#3) -- (m12) -- (g2) -- (m24) -- cycle;
  \draw[black!35,line width=0.22pt]
    (g1) -- (#2) (g1) -- (#3) (g1) -- (#4)
    (g1) -- (m12) (g1) -- (m13) (g1) -- (m23)
    (g2) -- (#2) (g2) -- (#3) (g2) -- (#5)
    (g2) -- (m12) (g2) -- (m14) (g2) -- (m24);
  \draw[#1,line width=0.8pt] (#2) -- (#3) -- (#4) -- cycle;
  \draw[#1,line width=0.8pt] (#2) -- (#3) -- (#5) -- cycle;
  \draw[#1,line width=1.9pt] (#2) -- (#3);
  \foreach \p in {#2,#3,#4,#5}{\fill[#1] (\p) circle (1.5pt);}
}

\usepackage[colorlinks,linkcolor=blue,citecolor=blue,urlcolor=blue]{hyperref}

\begin{document}

\maketitle

\begin{abstract}
  Calder\'on preconditioning restores mesh-independent conditioning for
  boundary-element discretizations of the Maxwell electric field integral
  equation, but assembling the dual barycentric operator can dominate the
  computational cost. We show that this preconditioner can be assembled once
  at a nominal geometry and reused across connectivity-preserving shape
  perturbations. Exact combinatorial alignment of the RWG--BC degrees of
  freedom transports the duality pairing, leaving only the dual-mesh operator
  frozen. Under uniform discrete-stability and geometric operator-continuity
  assumptions, a bi-parametric operator-preconditioning analysis yields a
  spectral-condition-number bound uniform in the mesh size and, after
  tensorization, in the stochastic-Galerkin dimension. Numerical experiments
  on smooth and non-smooth scatterers show that the frozen preconditioner
  closely tracks freshly assembled Calder\'on preconditioners over
  substantial shape variations. A Monte Carlo campaign on a NASA almond, whose
  admissibility is screened before each shape is constructed so that no
  realization is rejected, separates the two strategies by less than the variation between
  realizations, and mesh refinement there leaves the frozen iteration count
  unchanged. Reuse yields roughly $30\times$ measured speedup per new
  realization in our dense implementation; a cost-model estimate based on
  published compressed-assembly timings gives roughly $2\times$.
\end{abstract}

\begin{keywords}
  Calder\'on preconditioning, operator preconditioning, boundary element
  method, electric field integral equation, shape uncertainty
  quantification
\end{keywords}

\begin{AMS}
  65N38, 65F08, 65N12, 65R20, 78M15
\end{AMS}

\section{Introduction}
\label{sec:intro}

Time-harmonic electromagnetic scattering by a perfect electric conductor
(PEC) is classically reduced, via the electric field integral equation
(EFIE), to a boundary integral equation posed on the scatterer's
surface~$\Gamma$. Galerkin discretization by Rao--Wilton--Glisson (RWG)
elements yields dense, indefinite linear systems whose condition number
grows without bound as the mesh is refined (classically as $(\kappa
h)^{-2}$~\cite{AndriulliEtAl2008,AdrianAndriulliEibert2019}),
reflecting the different pseudodifferential orders of the EFIE's
solenoidal and non-solenoidal components. Calder\'on
(operator) preconditioning exploits the second-kind Calder\'on identity
$\calT \circ \calT = -\tfrac14 \mathcal I + \calK^2$, with $\calK$
compact on sufficiently smooth surfaces~\cite{ChristiansenNedelec2002},
to precondition the EFIE with itself, restoring mesh-independent
GMRES convergence~\cite{ChristiansenNedelec2002,Hiptmair2006}. In
practice this requires a \emph{dual}, div-conforming discretization
stable under the $L^2(\Gamma)$ pairing with the primal RWG space,
typically Buffa--Christiansen (BC) functions built on a barycentric
refinement of the original mesh~\cite{BuffaChristiansen2007}, as in the
multiplicative Calder\'on preconditioner
of~\cite{AndriulliEtAl2008}. Assembling the
preconditioning operator therefore costs several times more than assembling the
primal EFIE matrix itself (cf.~Section~\ref{sec:numanalysis}). Reducing this overhead has motivated a substantial
body of work, including refinement-free Calder\'on preconditioners that
eliminate the barycentric dual discretization
altogether~\cite{AdrianAndriulliEibert2019}, opposite-order and
bi-parametric operator
preconditioning~\cite{EscapilInchauspeJerezHanckes2021}, and $\mathcal
H$-matrix-compressed and reduced-accuracy Calder\'on
preconditioners~\cite{EscapilInchauspeJerezHanckes2019,
FierroJerezHanckes2020, KleanthousEtAl2022}, all of which reduce the
\emph{per-solve} cost of Calder\'on preconditioning for a single fixed
geometry and wavenumber.

Reusing a preconditioner across a sequence of related linear systems is a
classical idea in numerical linear algebra. Preconditioner updating constructs
cheap corrections to a factorization computed earlier in the
sequence~\cite{BenziBertaccini2003}, and Krylov
subspace recycling retains an approximate invariant subspace and redeploys it
across solves~\cite{ParksEtAl2006}. Both address sequences generated
algebraically, typically by a time step or a Newton update, in which the
matrix changes while the discrete space does not. Closest to the question
asked here, Graham, Pembery and Spence~\cite{GrahamPemberySpence2021} analyse
when the Galerkin matrix of one heterogeneous Helmholtz discretization
preconditions another, with wavenumber-explicit bounds on how far the
coefficients may move, and are likewise motivated by uncertainty
quantification; there the perturbation is in the material coefficients and
both matrices live on the same mesh. Here the perturbation is
geometric: the matrices live on different surfaces, and their
degree-of-freedom correspondence, once established, is an exact permutation
rather than an update formula. Neither technique has a counterpart
here: there is no factorization to correct and no subspace to carry over,
only the assembled dual-mesh operator, and it is that assembly which is paid
once and reused. Freezing is therefore the method here
rather than the baseline such schemes improve upon, which is what makes the
Calder\'on/BC setting the natural place to ask whether a preconditioner can
be frozen outright.

This paper addresses a question complementary to those per-solve cost
reductions, and in many-query settings equally important: given a family of
\emph{nearby} EFIE problems,
obtained by randomly or deterministically perturbing the scatterer's
shape, as arises in shape uncertainty quantification
(UQ)~\cite{AylwinJerezHanckesSchwabZech2020,
AylwinJerezHanckesSchwabZech2023, EscapilInchauspeJerezHanckes2024},
shape optimization, or inverse scattering, or by sweeping the
wavenumber, as arises in broadband simulation, must the Calder\'on
preconditioner be re-assembled at every query? In Monte Carlo and
multilevel domain-UQ pipelines for
computational
electromagnetics~\cite{AylwinJerezHanckesSchwabZech2020,
AylwinJerezHanckesSchwabZech2023}, hundreds to thousands of
geometrically nearby systems must be assembled and solved. The analysis
below covers both regimes: sampling methods
reuse $\bs P_0$ realization by realization, while intrusive
stochastic Galerkin formulations may use $I \otimes \bs P_0$ on the
coupled system (Corollary~\ref{cor:bochner}). For shape perturbations the
answer is no: a preconditioner assembled once at $\Gnom$ can be transported
by a degree-of-freedom (DOF) correspondence induced by the perturbation map
(Figure~\ref{fig:transport}) and reused unchanged on $\Gpert$, an operation
we call \emph{freezing} the Calder\'on preconditioner.

Once the perturbed systems are represented in a common nominal
discrete basis, reuse of the nominal preconditioner can be viewed as a
bi-parametric operator-preconditioning problem in the sense
of~\cite{EscapilInchauspeJerezHanckes2021}. What that framework needs of
the geometry is a single operator-norm estimate
(Assumption~\ref{ass:geocont}), so the argument applies to any parametric
family for which one is available. We focus on Calder\'on preconditioning
of the EFIE because, although expensive to assemble, it consistently outperforms
other EFIE preconditioners in iteration count and spectral
properties~\cite{EscapilInchauspeJerezHanckes2019,FierroPiccardoBetcke2023}. This
is precisely what the bi-parametric framework rewards: its bound inherits the
nominal constant $K_\star$ and degrades it only by the factor $(1+\nu)/(1-\nu)$,
so a well-conditioned nominal preconditioner, assembled once, remains
uniformly effective across the whole family. Van Harten and
Scarabosio~\cite{VanHartenScarabosio2026} recently considered the
complementary problem of placing several preconditioners across the
parameter domain for the Helmholtz equation.


\subsection*{Contributions}
We freeze the Calder\'on preconditioner: the dual-mesh operator is
assembled once at the nominal geometry and transported to each perturbed
geometry by a degree-of-freedom correspondence read off from the perturbation
map. Because an admissible perturbation moves mesh vertices without altering
connectivity, that correspondence is a permutation, so the transport is exact
and contributes no approximation of its own (Section~\ref{sec:implementation},
Lemma~\ref{lem:nxy-exact}). The numerics confirm it: the transported pairing
matrix agrees with the freshly assembled one to machine precision, while the
frozen dual-mesh operator does not.

Analytically, we identify such a perturbation of the Piola-pulled-back Maxwell
EFIE as an $(h,\nu)$-perturbation in the sense of the bi-parametric framework
of~\cite{EscapilInchauspeJerezHanckes2021}, with $\nu$ proportional to the
amplitude and constants uniform in $h$, and obtain from it an explicit bound
on the spectral condition number of the frozen-preconditioned system
(Theorem~\ref{thm:main}), uniform in the stochastic approximation dimension as
well for intrusive stochastic Galerkin discretizations
(Corollary~\ref{cor:bochner}). Numerically, we test the resulting frozen
preconditioner under shape, mesh and stochastic refinement on a sphere and a
triaxial ellipsoid, probe its boundary of validity on a Fichera corner, and
run a hundred-sample campaign on a three-wavelength NASA almond
(Section~\ref{sec:almond}), whose shapes are drawn from the admissible class
itself rather than prescribed. Across the smooth geometries freezing closely
tracks freshly assembled Calder\'on preconditioning; the Fichera experiments
identify perturbation of a reentrant singularity as the difficult case; and
the almond campaign tests admissibility, mesh refinement and many-query cost.
Freezing across \emph{wavenumber} rather than shape is a different problem,
and Theorem~\ref{thm:main} does not cover it: the penalty there is
asymmetric, saturating below the preconditioner's wavenumber and growing
without evident bound above it (Section~\ref{sec:kappamismatch}).

The companion work~\cite{EscapilInchauspeJerezHanckes2026} establishes
shape holomorphy of the pulled-back EFIE operator, together with the
parametric approximation rates that holomorphy yields. It says nothing about
preconditioning, about the conditioning of the discrete system, or about
reusing an operator across a family of geometries. From it we take a single
operator-norm estimate (Lemma~\ref{lem:shapepert}); everything downstream of
Assumption~\ref{ass:geocont} is established here.

\section{Mathematical framework}
\label{sec:framework}

\subsection{The electric field integral equation}
\label{sec:efie}
Let $D \subset \R^3$ be a bounded Lipschitz domain (the PEC scatterer)
with boundary $\Gamma := \partial D$ and exterior unit normal $\bs n$,
and let $D^c := \R^3 \setminus \overline D$ be the (unbounded) exterior
domain. For wavenumber $\kappa > 0$, the scattered field $\bs
U:D^c\to\C^3$ solves
\begin{align}
  \curl\curl \bs U(\bs x) - \kappa^2 \bs U(\bs x) &= \bs 0,
    &&\bs x \in D^c, \label{eq:helmholtz}\\
  \bs n \times \bs U &= -\bs n \times \bs U^{\mathrm{inc}},
    &&\bs x \in \Gamma, \label{eq:pec}\\
  \left|\curl \bs U(\bs x) \times \hat{\bs x} - \iota\kappa \bs
    U(\bs x)\right| &= o(\|\bs x\|^{-1}), && \|\bs x\| \to \infty,
    \label{eq:silvermuller}
\end{align}
where $\bs U^{\mathrm{inc}}$ is a known incident field and
\eqref{eq:silvermuller} is the Silver--M\"uller radiation condition.
Reducing~\eqref{eq:helmholtz}--\eqref{eq:silvermuller} to $\Gamma$ via
the free-space Green's function $$G_\kappa(\bs x,\bs y) :=
e^{-\iota\kappa\|\bs x-\bs y\|}/(4\pi\|\bs x-\bs y\|)$$ yields the EFIE:
find the surface current $\bs j$ such that
\begin{equation}
  \label{eq:efie}
  \begin{aligned}
  \calT(\bs j) &:= -\iota\kappa \bs n \times \int_\Gamma
    G_\kappa(\bs x,\bs y)\, \bs j(\bs y)\, d\bs y
  + \frac{1}{\iota\kappa} \bs n \times \nabla_{\bs x}
    \int_\Gamma G_\kappa(\bs x,\bs y)\, \dive_\Gamma \bs j(\bs y)\,
    d\bs y \\
  &\hphantom{:}= -\bs n \times \bs U^{\mathrm{inc}} =: \bs g.
  \end{aligned}
\end{equation}
Operator $\calT$ is a singular, non-local, first-kind boundary
integral operator whose solenoidal and non-solenoidal components have
different pseudodifferential orders; its Galerkin matrices, using RWG
basis functions, are dense and indefinite, and exhibit
dense-discretization breakdown. The condition number deteriorates
algebraically as $h \to 0$, classically as $(\kappa h)^{-2}$ in
coefficient norms~\cite{AndriulliEtAl2008,AdrianAndriulliEibert2019},
and deteriorates further as $\kappa \to 0$ or $\kappa \to \infty$.

Throughout the paper we assume that $\kappa$ is \emph{not} an interior
Maxwell resonance of the nominal scatterer $D_0$, i.e.\ that
$\kappa^2$ is not an interior Maxwell eigenvalue of $D_0$; this
excludes the classical interior-resonance failure of the EFIE at
discrete real wavenumbers and guarantees that the nominal operator
$\calT$ is injective, hence (by Fredholm theory on the trace spaces of
Section~\ref{sec:shapebound}) invertible, so that the nominal
\emph{continuous} inf-sup condition holds~\cite{BuffaHiptmair2003}. The
corresponding $h$-uniform \emph{discrete} inf-sup constants are not claimed
to follow from this alone; they are part of
Assumption~\ref{ass:stability}. For the perturbed family,
the geometric continuity estimate of Assumption~\ref{ass:geocont}
(Section~\ref{sec:shapebound}) and the perturbation framework of
Section~\ref{sec:oppg} then guarantee persistence of the
\emph{discrete} inf-sup stability required below for all sufficiently
small perturbation amplitudes, so no separate discrete non-resonance
assumption is needed for each perturbed geometry.

\subsection{Calder\'on preconditioning}
\label{sec:calderon}
At the continuous level, the operator $\calT$ satisfies the second-kind Calder\'on
identity
\begin{equation}
  \label{eq:calderon-identity}
  (\calT \circ \calT)(\bs j) = \calT^2(\bs j)
    = -\frac14 \bs j + \calK^2(\bs j), \qquad
  \calK(\bs j) := \bs n \times \nabla_{\bs x} \int_\Gamma
    G_\kappa(\bs x,\bs y)\, \bs j(\bs y)\, d\bs y.
\end{equation}
The identity itself holds on Lipschitz surfaces; the
compact-perturbation-of-identity interpretation, namely that $\calK$ is
compact, so that $\calT^2$ is a second-kind operator, additionally requires
sufficient smoothness of $\Gamma$, and is classical for smooth
surfaces~\cite{ChristiansenNedelec2002}. On merely Lipschitz or
polyhedral surfaces $\calK$ is in general not compact, and Calder\'on
preconditioning is then best understood through the operator
preconditioning framework of~\cite{Hiptmair2006} (mesh-independent
bounds without a second-kind interpretation), which is the viewpoint
adopted throughout this paper. Preconditioning~\eqref{eq:efie} by (a
discretization of) $\calT$ itself yields a well-conditioned system.
Discretely, stability of the resulting pairing requires a
\emph{dual} basis, i.e.~Buffa--Christiansen functions defined on the
barycentric refinement of the primal mesh~\cite{BuffaChristiansen2007}.
Writing $X_h$ for the RWG space and $Y_h$ for the BC space
on a mesh of size $h$, one assembles $\bs Z_{XX} := \bs A_h[X_h,X_h]$
(the primal EFIE matrix), $\bs T_{YY} := \bs A_h[Y_h,Y_h]$ (the dual
discretization), and $\bs N_{XY}$, the $X_h$--$Y_h$ duality pairing, and
solves $\bs Z_{XX} \bs u = \bs b$ with left preconditioner $\bs P :=
\bs N_{XY}^{-\top}\bs T_{YY}\bs N_{XY}^{-1}$, applied via inner GMRES
solves against $\bs N_{XY}$. Because the barycentric mesh has six times
as many triangles as the primal mesh, assembling $\bs T_{YY}$ typically
costs several times more than $\bs Z_{XX}$: for our sphere test case at
$h=0.1$, $\kappa=2$ ($4827$ degrees of freedom), $\bs T_{YY}$ alone
accounts for 430\,s out of 442\,s of total assembly time, 97\%
(Table~\ref{tab:componenttiming}). This is the cost the present paper
seeks to amortize across many queries rather than reduce per query.

\subsection{Bi-parametric operator preconditioning}
\label{sec:oppg}
Our results build on the operator preconditioning (Petrov--Galerkin,
OP-PG) framework of~\cite{Hiptmair2006}, extended to \emph{bi-parametric}
perturbed forms by~\cite{EscapilInchauspeJerezHanckes2021}, to which we refer
for full statements and proofs. Let $X,Y,V,W$ be reflexive Banach spaces, let
$\mathsf a \in \mathcal L(X\times Y;\C)$ have discrete inf-sup constant
$\gamma_{\mathsf a}$ and continuity constant $\|\mathsf a\|$, and let $\mathsf
c \in \mathcal L(V \times W;\C)$, with pairings $\mathsf n \in \mathcal L(V
\times Y;\C)$ and $\mathsf m \in \mathcal L(X \times W;\C)$, induce the
preconditioned operator equation
\begin{equation}
  \label{eq:op-pg}
  \text{find } u \in X \text{ such that } \mathrm{P}\mathrm{A} u = \mathrm{P} b,
  \qquad \mathrm{P} := \mathrm{M}^{-1}\mathrm{C}\mathrm{N}^{-1}.
\end{equation}

\begin{definition}[$(h,\nu)$-perturbation]
  \label{def:perturbation}
  (\cite[Def.~1]{EscapilInchauspeJerezHanckes2021})
  Let $\nu \in [0,1)$ and $h>0$. A form $\mathsf a_\nu \in \mathcal
  L(X\times Y;\C)$ is an $(h,\nu)$-perturbation of $\mathsf a$,
  $\mathsf a_\nu \in \Phi_{h,\nu}(\mathsf a)$, if
  \begin{equation}
    \gamma_{\mathsf a}^{-1} \left| \mathsf a(u_h,v_h) - \mathsf
    a_\nu(u_h,v_h) \right| \le \nu \|u_h\|_X \|v_h\|_Y,
    \qquad \forall\, u_h \in X_h,\ v_h \in Y_h.
  \end{equation}
\end{definition}

For $\mu,\nu \in [0,1)$ we take $\mathsf c_\mu \in \Phi_{h,\mu}(\mathsf
c)$ and $\mathsf a_\nu \in \Phi_{h,\nu}(\mathsf a)$, and write $\mathrm{P}_\mu
\mathrm{A}_\nu$ for the bi-parametric preconditioned operator.

\begin{theorem}[Bi-parametric operator preconditioning]
  \label{thm:ejh-op}
  (\cite[Thm.~2]{EscapilInchauspeJerezHanckes2021})
  For $\mu,\nu \in [0,1)$ and $h>0$, the spectral condition number of
  the bi-parametric preconditioned matrix satisfies
  \begin{equation}
    \label{eq:kappastar}
    \kappa_S(\bs P_\mu \bs A_\nu) \le K_\star
      \left(\frac{1+\mu}{1-\mu}\right)\left(\frac{1+\nu}{1-\nu}\right)
      =: K_{\star,\mu,\nu}, \qquad
    K_\star := \frac{\|\mathsf m\|\,\|\mathsf n\|\,\|\mathsf c\|\,
      \|\mathsf a\|}{\gamma_{\mathsf m}\gamma_{\mathsf n}\gamma_{\mathsf
      c}\gamma_{\mathsf a}}.
  \end{equation}
\end{theorem}

Here and throughout, $\kappa_S$ is the spectral condition number in the
sense of~\cite[eq.~(20)]{EscapilInchauspeJerezHanckes2021}, namely
\begin{equation}
  \label{eq:kappaS-def}
  \kappa_S(\bs M) \;:=\; \varrho(\bs M)\,\varrho(\bs M^{-1})
  \;=\; \frac{|\lambda_{\max}(\bs M)|}{|\lambda_{\min}(\bs M)|},
\end{equation}
with $\varrho$ the spectral radius. It agrees with the Euclidean
$\kappa_2$ only for normal $\bs M$, so a bound on $\kappa_S$ does not by
itself guarantee Krylov convergence, whence the separate hypothesis of
Theorem~\ref{thm:ejh-gmres}. The Euclidean counterpart
of~\eqref{eq:kappastar}, $\kappa_2 (\bs P_\mu \bs A_\nu)\le K_{\star,\mu,\nu}K_{\Lambda_h}^2$,
carries the conditioning $K_{\Lambda_h}$ of the synthesis operator of $X_h$,
$h$-dependent for RWG-type
spaces~\cite[eq.~(45)]{EscapilInchauspeJerezHanckes2021}; the $h$-uniform
statement is therefore the one for $\kappa_S$.

\begin{theorem}[GMRES$(m)$ linear convergence]
  \label{thm:ejh-gmres}
  (\cite[Thm.~4]{EscapilInchauspeJerezHanckes2021})
  Under coercivity-type
  hypothesis~\cite[Assumption~4]{EscapilInchauspeJerezHanckes2021}
  for the bi-parametric
  preconditioned system \eqref{eq:op-pg}, the weighted GMRES$(m)$
  residual factor after $k$ steps, $1 \le k,m \le N$, satisfies
  \begin{equation}
    \label{eq:gmres-bound}
    \Theta_k^{(m)} \le \left(1 - \frac{1}{K_{\star,\mu,\nu}}\right)^{1/2}.
  \end{equation}
\end{theorem}

The constants entering $K_\star$ are \emph{discrete} quantities attached
to each mesh level $h$, so~\eqref{eq:kappastar}--\eqref{eq:gmres-bound} are
$h$-uniform only if they are. We make that a standing hypothesis.

\begin{sgassumption}{S}{standing assumption: $h$-uniform discrete stability}
  \label{ass:stability}
  There exist $h_0 > 0$ and constants $\underline\gamma > 0$,
  $\overline M < \infty$ such that, for all $0 < h \le h_0$, the
  nominal Galerkin families and duality pairings satisfy
  \begin{equation}
    \min\{\gamma_{\mathsf m}, \gamma_{\mathsf n}, \gamma_{\mathsf c},
    \gamma_{\mathsf a_0}\} \;\ge\; \underline\gamma, \qquad
    \max\{\|\mathsf m\|, \|\mathsf n\|, \|\mathsf c\|, \|\mathsf
    a_0\|\} \;\le\; \overline M .
  \end{equation}
  In particular $K_\star \le (\overline M/\underline\gamma)^4$
  uniformly in $h \le h_0$, and $\underline\gamma_{\mathsf
  a_0} := \inf_{0<h\le h_0} \gamma_{\mathsf a_0,h} \ge
  \underline\gamma$.
\end{sgassumption}

Assumption~\ref{ass:stability} collects the uniform stability properties
that the abstract OP-PG framework requires, and we retain them as a
hypothesis rather than reproduce their proofs. The abstract constants have
concrete counterparts here. The pair $(\gamma_{\mathsf n},\|\mathsf n\|)$,
and its transpose $(\gamma_{\mathsf m},\|\mathsf m\|)$, are those of the
RWG--BC duality pairing realized by $\bs N_{XY}$; their $h$-uniformity
follows from the dual finite element complex on the barycentric
refinement~\cite{BuffaChristiansen2007}. The pairs $(\gamma_{\mathsf
a_0},\|\mathsf a_0\|)$ and $(\gamma_{\mathsf c},\|\mathsf c\|)$ are those of
the nominal EFIE Galerkin form on $X_h \times Y_h$ and on $Y_h$
respectively; away from interior Maxwell resonances, uniform discrete
stability for such forms holds under the conforming-approximation and
discrete-compactness hypotheses underlying Maxwell
BEM~\cite{BuffaHiptmair2003}. We do not assert that the generalized
G\aa rding inequality alone yields the $h$-uniform discrete inf-sup constants
in the precise Petrov--Galerkin configuration used here; that is part of what
Assumption~\ref{ass:stability} assumes. It is in any case the standard
hypothesis under which operator preconditioning delivers mesh-independent
bounds~\cite{Hiptmair2006}, not a restriction peculiar to this paper.

Theorems~\ref{thm:ejh-op} and~\ref{thm:ejh-gmres} were originally applied
in~\cite{EscapilInchauspeJerezHanckes2021} to \emph{numerical}
perturbations, such as coarser quadrature and low-rank compression, of an
otherwise fixed variational problem. 
The first step here is to recognize
that a \emph{geometric} perturbation of the scatterer, for a fixed
discretization strategy, induces the same $(h,\nu)$-perturbation
structure on the EFIE bilinear form, with $\nu$ proportional to the
perturbation amplitude and with constants uniform in $h$; we make this
precise next.

\section{Robustness of the frozen preconditioner to shape perturbation}
\label{sec:theory}

\subsection{Admissible shape perturbations}
\label{sec:shapeperts}
The parametric framework below is formulated on Lipschitz boundaries;
additional smoothness enters only in Lemma~\ref{lem:shapepert}. We
follow the nominal-domain parametrization that is standard in shape UQ
for computational
electromagnetics~\cite{AylwinJerezHanckesSchwabZech2020,
AylwinJerezHanckesSchwabZech2023}; see
also~\cite{DoelzHenriquez2024} for the same affine-parametric
boundary description in the boundary-integral context. Let $\Gnom =
\partial D_0$ be
a fixed nominal Lipschitz boundary, let $\{\bs\psi_j\}_{j\in\N} \subset
C^{1,1}(\R^3;\R^3)$ be a summable sequence of deformation fields,
$\sum_j \|\bs\psi_j\|_{C^{1,1}} < \infty$, and let $\bs y \in U$, with
$U := [-1,1]^{\N}$, be a \emph{deterministic} parameter sequence; all
estimates below hold uniformly over $\bs y \in U$. (For the
uncertainty-quantification interpretation, $U$ may additionally be
equipped with the product uniform probability measure, so that the
$y_j$ become i.i.d.\ $\mathrm{Unif}(-1,1)$ random variables and every
uniform-in-$\bs y$ statement holds in particular almost surely; none
of the analysis requires this probabilistic structure.) Define the
perturbed boundary map and its induced boundary
\begin{equation}
  \label{eq:pertmap}
  \bs\varphi_{\bs y} := \mathrm{Id} + \varepsilon \sum_{j\in\N} y_j \bs\psi_j :
    \Gnom \to \R^3, \qquad
  \Gpert := \bs\varphi_{\bs y}(\Gnom),
\end{equation}
where $\varepsilon \ge 0$ is the (deterministic) perturbation
\emph{amplitude}. By the summability assumption, $\|\bs\varphi_{\bs y} -
\mathrm{Id}\|_{C^{1,1}(\Gnom)} \le C_0\varepsilon$ for every $\bs y
\in U$, with $C_0 := \sum_j \|\bs\psi_j\|_{C^{1,1}}$ independent of $\bs
y$. We assume $\varepsilon$ is small enough that $\bs\varphi_{\bs y}$ is,
for every $\bs y \in U$, a bi-Lipschitz diffeomorphism onto its image
(guaranteed for $\varepsilon < 1/(2C_0)$ by a standard Neumann-series
argument on $d\bs\varphi_{\bs y}$), so that $\Gpert$ is itself a Lipschitz
boundary and mesh connectivity is preserved: a triangulation $\{\Gnom^{(i)}\}$ of $\Gnom$ induces, via
$\bs\varphi_{\bs y}$, a triangulation $\{\Gpert^{(i)} := \bs\varphi_{\bs y}(\Gnom^{(i)})\}$ of $\Gpert$ with \emph{identical connectivity}:
only vertex coordinates move. Section~\ref{sec:implementation} exploits
this at the implementation level: while the RWG and BC basis
\emph{functions} themselves depend on the vertex geometry, their
canonical \emph{indexing}, which primal-mesh edge each basis
function is attached to, is determined by mesh connectivity alone, so
connectivity-preserving deformations induce an exact, canonical
correspondence between the degrees of freedom of $X_h(\Gnom)$ and
$X_h(\Gpert)$ (and likewise for $Y_h$), realized in practice by a
permutation matrix, with no geometric interpolation or re-meshing
required.

\subsection{A shape-perturbation bound for the EFIE bilinear form}
\label{sec:shapebound}

We first instantiate the abstract spaces of Section~\ref{sec:oppg} for
the EFIE. On the nominal boundary, the energy space of the EFIE is the
tangential trace space
\begin{equation}
  \label{eq:tracespace}
  X \;:=\; \bs H^{-1/2}(\dive_\Gamma, \Gnom)
\end{equation}
of tangential fields with components in $\bs H^{-1/2}(\Gnom)$ and
surface divergence in the space $H^{-1/2}(\Gnom)$~\cite{BuffaCostabelSheen2002},
on which $\calT: X \to X'$ is bounded and satisfies a generalized
G\aa rding inequality~\cite{BuffaHiptmair2003}; the dual $X'$ is
canonically identified with $\bs H^{-1/2}(\mathbf{curl}_\Gamma,\Gnom)$
under the anti-symmetric rotated duality pairing $\langle \bs u, \bs
v\rangle_{\times,\Gnom} := \int_{\Gnom} (\bs u \times \bs n)\cdot \bs
v$. We take $Y := X$ and define the nominal EFIE bilinear form
\begin{equation}
  \mathsf a_0(\bs u,\bs v) := \langle \calT_{\Gnom}\bs u, \bs
  v\rangle_{\times,\Gnom}, \qquad \bs u \in X,\ \bs v \in Y.
\end{equation}
Let $X_h, Y_h \subset X$ denote div-conforming RWG/BC-type spaces on
an exact, surface-fitted triangulation of $\Gnom$: the RWG space $X_h$
on the primal mesh and the BC space $Y_h$ on its barycentric
refinement~\cite{BuffaChristiansen2007}, understood, on curved
elements, as the Piola images of the reference RWG/BC shape functions
(so that the fields are piecewise polynomial in reference coordinates,
though not generally in physical coordinates). Only conformity in $X$
is used below; no finite-dimensionality or inverse-estimate argument
enters, which is what makes the resulting constants independent of
$h$. Our \emph{implementation},
however, uses straight-sided triangles, whose RWG/BC functions live on
a polyhedral approximation $\Gamma_{0,h}$ of the smooth surface
$\Gnom$ rather than on $\Gnom$ itself, and therefore sits outside this
literal conforming setting; Remark~\ref{rmk:variationalcrime}
discusses the status of that distinction.

Next we pull the perturbed problem back to these fixed nominal spaces.
For $r \in [0,\varepsilon]$ set $\bs\varphi_r := \mathrm{Id} + r\bs V$ with
fixed direction $\bs V := \sum_j y_j \bs\psi_j$, $\Gamma_r :=
\bs\varphi_r(\Gnom)$ (so $\Gamma_\varepsilon = \Gpert$), and let
\begin{equation}
  \calP_r : \bs H^{-1/2}(\dive_\Gamma,\Gnom) \longrightarrow
            \bs H^{-1/2}(\dive_\Gamma,\Gamma_r)
\end{equation}
be the contravariant (Piola) surface transform associated with
$\bs\varphi_r$. This is the standard div-conformity-preserving transport, which is
a uniformly bounded isomorphism for $r\|\bs V\|_{C^{1,1}}$ small
\cite{CostabelLeLouer2012b}. Let $\calP_r' : \bs H^{-1/2}(\mathbf{curl}_\Gamma,\Gamma_r) \to \bs
H^{-1/2}(\mathbf{curl}_\Gamma,\Gnom)$ be the transport dual to $\calP_r$ with
respect to the rotated Maxwell pairings, defined by
\begin{equation}
  \label{eq:piola-dual}
  \big\langle \calP_r'\bs f, \bs v\big\rangle_{\times,\Gnom}
  \;=\; \big\langle \bs f, \calP_r\bs v\big\rangle_{\times,\Gamma_r}
  \qquad \forall\, \bs v \in X.
\end{equation}
Define the pulled-back operator and form
\begin{equation}
  \label{eq:pulledback}
  \widetilde\calT_r := \calP_r' \circ \calT_{\Gamma_r} \circ \calP_r
  \;\in\; \calL(X, X'),
\end{equation}
and
\begin{equation}
  \widetilde{\mathsf a}_r(\bs u,\bs v)
    := \big\langle \widetilde\calT_r \bs u, \bs v
       \big\rangle_{\times,\Gnom}
    = \big\langle \calT_{\Gamma_r}(\calP_r \bs u), \calP_r\bs v
       \big\rangle_{\times,\Gamma_r},
\end{equation}
the second equality above being
exactly~\eqref{eq:piola-dual} applied to $\bs f = \calT_{\Gamma_r}(\calP_r\bs
u)$. The two transports are compatible with the rotated Maxwell duality:
combining the contravariant Piola transformation with Nanson's formula and
the surface change of variables gives $\langle \calP_r\bs u, \calP_r\bs
v\rangle_{\times,\Gamma_r} = \langle \bs u,\bs v\rangle_{\times,\Gnom}$.
Hence $\calP_r'$ is a bounded isomorphism, uniformly for $r$ small, and
$\widetilde\calT_r \in \calL(X,X')$ as stated. This is the pullback
of~\cite{EscapilInchauspeJerezHanckes2026}, where the cancellation is shown to
be exact: no unit normal, no inverse surface metric and no surface Jacobian
survives in the pulled-back form, only the deformed kernel and the first
tangential differential~\cite[Prop.~4.1]{EscapilInchauspeJerezHanckes2026}.
Lemma~\ref{lem:nxy-exact} is the discrete counterpart of that cancellation for
the RWG--BC pairing.

By construction $\widetilde{\mathsf a}_0 = \mathsf a_0$, and we write
$\mathsf a_{\bs y} := \widetilde{\mathsf a}_\varepsilon$. The whole family $\{\widetilde\calT_r\}_r$ acts between the
\emph{fixed} spaces $X$ and $X'$ on the \emph{fixed} boundary $\Gnom$.
Because the Piola transform of a nominal-mesh RWG (resp.\ BC) basis
function associated with a connectivity-preserving mesh deformation is
the corresponding deformed-mesh basis function, the Galerkin matrix of
$\mathsf a_{\bs y}$ on $X_h \times Y_h$, expressed in the DOF-aligned
nominal basis of Section~\ref{sec:implementation}, is the
perturbed-geometry EFIE matrix used in practice (up to the
piecewise-affine geometry approximation discussed in
Remark~\ref{rmk:variationalcrime}).

The reuse argument requires one geometric hypothesis.

\begin{sgassumption}{G}{geometric operator continuity}
  \label{ass:geocont}
  There exist $\varepsilon_0 > 0$ and $C_{\mathrm{geo}} > 0$ such that
  the pulled-back EFIE operators of~\eqref{eq:pulledback} satisfy
  \begin{equation}
    \label{eq:geocont}
    \big\| \widetilde\calT_\varepsilon - \widetilde\calT_0
    \big\|_{\calL(X,X')} \;\le\; C_{\mathrm{geo}}\,\varepsilon,
    \qquad 0 \le \varepsilon \le \varepsilon_0,
  \end{equation}
  uniformly with respect to $\bs y \in U$.
\end{sgassumption}

Assumption~\ref{ass:geocont} is the only geometry-dependent input to
Theorem~\ref{thm:main}; boundary regularity plays no role there.
It enters only in the verification that follows.
Reference~\cite{EscapilInchauspeJerezHanckes2026} establishes
operator-valued holomorphy of the Piola-pulled-back Maxwell EFIE on a fixed
$C^{1,1}$ reference boundary, in the trace spaces~\eqref{eq:tracespace} and
under the pullback~\eqref{eq:pulledback}. We do not use the holomorphy itself,
only the Lipschitz estimate it yields through a Cauchy estimate along complex
segments, \cite[Cor.~6.1]{EscapilInchauspeJerezHanckes2026}: writing
$\widetilde\calT_{\bs y}$ for the pullback~\eqref{eq:pulledback} at parameter
$\bs y$ and $b_j$ for the $W^{2,\infty}(\Gnom)$-norm of the $j$th deformation
field,
\begin{equation}
  \label{eq:corA}
  \big\|\widetilde\calT_{\bs y} - \widetilde\calT_{\tilde{\bs y}}
  \big\|_{\calL(X,X')} \;\le\; C_L \sum_{j\in\N} b_j\,|y_j - \tilde y_j|,
  \qquad \bs y, \tilde{\bs y} \in U,
\end{equation}
with $C_L$ independent of the parameters. Assumption~\ref{ass:geocont} is the
case $\tilde{\bs y} = \bs 0$.

\begin{lemma}[verification of Assumption~\ref{ass:geocont} for $C^{1,1}$
  geometries]
  \label{lem:shapepert}
In addition to the setting of Section~\ref{sec:shapeperts}, let $\Gnom$
  be of class $C^{1,1}$ and let the deformation fields $\{\bs\psi_j\}_{j\in\N}
  \subset C^{1,1}(\R^3;\R^3)$ satisfy the summability hypothesis of
  Section~\ref{sec:shapeperts}, $C_0 := \sum_j\|\bs\psi_j\|_{C^{1,1}} <
  \infty$, and let the $W^{2,\infty}(\Gnom)$-norms of their restrictions
  to $\Gnom$ lie in $\ell^p(\N)$ for some $0 < p < 1$. Then
  Assumption~\ref{ass:geocont} holds: there exist
  $\varepsilon_0 \in (0,1/(2C_0))$ and a bounded $C_{\mathrm{geo}}$,
  depending on $\kappa$, on $\Gnom$, on $\varepsilon_0$ and on the
  amplitude sequence $(\|\bs\psi_j\|_{W^{2,\infty}(\Gnom)})_{j}$ itself ---
  which enters the admissibility budget
  of~\cite{EscapilInchauspeJerezHanckes2026} and not only through its sum
  $C_0$ --- both independent of $\bs y
  \in U$ and of any discretization, such that~\eqref{eq:geocont} holds.
\end{lemma}

\begin{proof}
 \emph{Step 1 (uniformity of the transport).} Let $\varepsilon_0 C_0 <
  1/2$. For $r \in [0,\varepsilon_0]$ the map $\bs\varphi_r = \mathrm{Id} + r\bs
  V$ satisfies $\|\bs\varphi_r - \mathrm{Id}\|_{C^{1,1}} \le r\|\bs
  V\|_{C^{1,1}} \le rC_0 < 1/2$ for every $\bs y \in U$, so $d\bs\varphi_r$ is
  invertible with a Neumann-series bound and $\bs\varphi_r$ is a $C^{1,1}$
  diffeomorphism onto its image. The associated transports $\calP_r$,
  $\calP_r^{-1}$ and $\calP_r'$ of~\eqref{eq:piola-dual} are then bounded
  uniformly in $r \in [0,\varepsilon_0]$ and $\bs y \in U$, and the
  pulled-back operators~\eqref{eq:pulledback} form a family in the
  \emph{fixed} space $\calL(X,X')$ with $\widetilde\calT_0 = \calT_{\Gnom}$.

  \emph{Step 2 (import).} Fix $\varepsilon \le \varepsilon_0$ and apply
  the parametric framework of~\cite{EscapilInchauspeJerezHanckes2026} to the
  \emph{fixed} family with deformation fields $\varepsilon_0\,\bs\psi_j
  |_{\Gnom} \in W^{2,\infty}(\Gnom;\R^3)$, its own fields up to that scaling.
  Its amplitudes
  $b_j = \varepsilon_0\|\bs\psi_j|_{\Gnom}\|_{W^{2,\infty}(\Gnom)} \le
  c\,\varepsilon_0\|\bs\psi_j\|_{C^{1,1}}$ lie in $\ell^p(\N)$ by hypothesis
  and satisfy $\sum_j b_j \le c\,\varepsilon_0 C_0$, while Step~1 supplies the uniform
  bi-Lipschitz bound; together with $\Gnom \in C^{1,1}$ these
  are~\cite[Assump.~2.1--2.2]{EscapilInchauspeJerezHanckes2026}. The map
  $\bs\varphi_{\bs y}$ of~\eqref{eq:pertmap} at amplitude $\varepsilon$ is that
  family evaluated at the parameter $(\varepsilon/\varepsilon_0)\bs y \in U$,
  so~\eqref{eq:corA} with $\tilde{\bs y} = \bs 0$ gives, uniformly in $\bs y
  \in U$,
  \begin{equation}
    \label{eq:derivbound}
    \big\|\widetilde\calT_\varepsilon - \widetilde\calT_0
      \big\|_{\calL(X,X')}
    \;\le\; C_L\,\frac{\varepsilon}{\varepsilon_0} \sum_{j\in\N} b_j
    \;\le\; c\,C_L C_0\,\varepsilon,
  \end{equation}
  which is~\eqref{eq:geocont} with $C_{\mathrm{geo}} := c\,C_L C_0$.
\end{proof}

\begin{corollary}[discrete $(h,\nu)$-perturbation]
  \label{cor:hnu}
  Let Assumption~\ref{ass:geocont} hold. Then, for all $0 \le
  \varepsilon \le \varepsilon_0$,
  \begin{equation}
    \label{eq:shapepertbound}
    \left| \mathsf a_0(\bs u_h,\bs v_h) - \mathsf a_{\bs y}(\bs u_h,\bs
    v_h) \right| \le C_{\mathrm{geo}}\,\varepsilon\, \|\bs u_h\|_X
    \|\bs v_h\|_Y,
    \qquad \forall\, \bs u_h \in X_h,\ \bs v_h \in Y_h,
  \end{equation}
  uniformly in $\bs y \in U$ and with a constant independent of $h$.
  Equivalently, writing $\gamma_{\mathsf a_0,h}$ for the discrete
  nominal inf-sup constant at mesh level $h$, $\mathsf a_{\bs y} \in
  \Phi_{h,\nu_h}(\mathsf a_0)$ (Definition~\ref{def:perturbation})
  with the mesh-specific parameter
  \begin{equation}
    \label{eq:nu-def}
    \nu_h(\varepsilon) := \frac{C_{\mathrm{geo}}\,\varepsilon}{
      \gamma_{\mathsf a_0,h}},
  \end{equation}
  which lies in $[0,1)$ for all $\varepsilon < \gamma_{\mathsf
  a_0,h}/C_{\mathrm{geo}}$. Under Assumption~\ref{ass:stability},
  setting
  \begin{equation}
    \label{eq:nubar-def}
    \bar\nu(\varepsilon) := \frac{C_{\mathrm{geo}}\,\varepsilon}{
      \underline\gamma_{\mathsf a_0}} \;\ge\; \nu_h(\varepsilon),
  \end{equation}
  one has $\mathsf a_{\bs y} \in \Phi_{h,\nu_h}(\mathsf a_0) \subset
  \Phi_{h,\bar\nu}(\mathsf a_0)$ for every $h \le h_0$, with
  $\bar\nu(\varepsilon)$ independent of $h$; admissibility
  ($\bar\nu < 1$) holds simultaneously at every mesh level once
  $\varepsilon < \underline\gamma_{\mathsf a_0}/C_{\mathrm{geo}}$.
\end{corollary}

\begin{proof}
  For $\bs u_h \in X_h \subset X$ and $\bs v_h \in Y_h \subset X$
  (conformity, Section~\ref{sec:shapebound}),
  \begin{equation*}
    |\mathsf a_{\bs y}(\bs u_h,\bs v_h) - \mathsf a_0(\bs u_h,\bs v_h)|
    = \big|\big\langle (\widetilde\calT_\varepsilon -
      \widetilde\calT_0)\bs u_h, \bs v_h\big\rangle_{\times,\Gnom}\big|
    \le C_{\mathrm{geo}}\varepsilon\, \|\bs u_h\|_X\|\bs v_h\|_X
  \end{equation*}
  by~\eqref{eq:geocont}. The constant is independent of $h$ because
  no property of the discrete spaces is used other than containment in
  $X$. Dividing by the discrete inf-sup constant $\gamma_{\mathsf
  a_0,h}$ of the nominal form yields the $(h,\nu_h)$-perturbation
  statement; the inclusion $\Phi_{h,\nu_h} \subset \Phi_{h,\bar\nu}$
  is immediate from $\nu_h \le \bar\nu$ in
  Definition~\ref{def:perturbation}.
\end{proof}

\subsection{Main theorem}
We now adapt Theorems~\ref{thm:ejh-op}--\ref{thm:ejh-gmres} to the
frozen-preconditioner strategy: the preconditioner ingredients $\bs
T_{YY,0}, \bs N_{XY,0}$ are assembled \emph{exactly} at the nominal
geometry $\Gnom$ (no numerical perturbation, $\mu = 0$), and reused,
via the DOF-alignment permutation of
Section~\ref{sec:implementation}, to
precondition the EFIE assembled at the perturbed geometry $\Gpert$
(perturbation parameters $\nu_h(\varepsilon)$ and
$\bar\nu(\varepsilon)$ from Corollary~\ref{cor:hnu}).

\begin{theorem}[Robustness of the frozen Calder\'on preconditioner]
  \label{thm:main}
  Let
  \begin{equation*}
    \bs P_0 := \bs N_{XY,0}^{-\top}\,\bs T_{YY,0}\,\bs N_{XY,0}^{-1}
  \end{equation*}
  be the Calder\'on preconditioner assembled exactly at the nominal
  geometry $\Gnom$, and let $\bs A_{\varepsilon,h}$ denote the EFIE
  matrix assembled at a $\varepsilon$-amplitude shape perturbation
  $\Gpert$, expressed in the DOF-aligned nominal basis. 
  
  Let
  Assumption~\ref{ass:stability} ($h$-uniform discrete stability) and
  Assumption~\ref{ass:geocont} (geometric operator continuity) hold,
  and let $0 < h \le h_0$ and $0 \le \varepsilon \le \varepsilon_0$ with
  $\varepsilon < \underline\gamma_{\mathsf a_0}/C_{\mathrm{geo}}$, so that
  $\nu_h(\varepsilon) \le \bar\nu(\varepsilon) \in [0,1)$ by
  Corollary~\ref{cor:hnu}, uniformly in $h$. Then the spectral
  condition number of the frozen-preconditioned perturbed system
  satisfies
  \begin{equation}
    \label{eq:main-cond}
    \kappa_S(\bs P_0 \bs A_{\varepsilon,h}) \;\le\; K_\star \cdot
      \frac{1+\nu_h(\varepsilon)}{1-\nu_h(\varepsilon)}
    \;\le\; \bar K_\star \cdot
      \frac{1+\bar\nu(\varepsilon)}{1-\bar\nu(\varepsilon)},
    \qquad
    \bar K_\star := (\overline M/\underline\gamma)^4,
  \end{equation}
  where the first bound is the sharper mesh-specific estimate and the
  second is independent of $h \le h_0$. If, in addition,
  the coercivity-type
  hypothesis~\cite[Assumption~4]{EscapilInchauspeJerezHanckes2021}
  holds for the resulting preconditioned system, then the weighted
  GMRES$(m)$ residual factor after $k$ steps satisfies
  \begin{equation}
    \label{eq:main-gmres}
    \Theta_k^{(m)} \le \left(1 - \frac{1-\bar\nu(\varepsilon)}{\bar
      K_\star (1+\bar\nu(\varepsilon))}\right)^{1/2}
    \xrightarrow[\varepsilon \to 0]{} \left(1 -
      \frac1{\bar K_\star}\right)^{1/2},
  \end{equation}
  recovering, in the limit of no perturbation, the
  unperturbed-geometry GMRES bound obtained from
  Theorem~\ref{thm:ejh-gmres} at $\mu=\nu=0$. In particular,
  Lemma~\ref{lem:shapepert} verifies Assumption~\ref{ass:geocont} for
  the $C^{1,1}$ Maxwell geometries considered in this paper.
\end{theorem}

\begin{proof}
  Apply Theorem~\ref{thm:ejh-op} with $\mu = 0$ (so $\mathsf c_0 =
  \mathsf c$, i.e.\ $\bs T_{YY,0}, \bs N_{XY,0}$ are unperturbed) and
  $\nu = \nu_h(\varepsilon)$ as furnished by Corollary~\ref{cor:hnu}
  for $\mathsf a_{\bs y} \in \Phi_{h,\nu_h(\varepsilon)}(\mathsf
  a_0)$; since $(1+0)/(1-0) = 1$, \eqref{eq:kappastar} collapses to
  the first inequality in~\eqref{eq:main-cond}. The second follows
  from $\nu_h \le \bar\nu$, the monotonicity of $t \mapsto
  (1+t)/(1-t)$, and $K_\star \le \bar K_\star$ under
  Assumption~\ref{ass:stability}; every quantity in it is uniform in
  $h \le h_0$. Under the additional coercivity-type hypothesis,
  substituting $\mu=0,\ \nu=\bar\nu(\varepsilon)$ into
  Theorem~\ref{thm:ejh-gmres} gives~\eqref{eq:main-gmres}.
\end{proof}

\subsection{Scope of the theoretical result}
\label{sec:scope}
 What Theorem~\ref{thm:main} delivers without the additional
coercivity-type hypothesis is the following: a geometric perturbation of amplitude $\varepsilon$ enters as an
$(h,\nu)$-perturbation with $\nu$ proportional to $\varepsilon$, and the
resulting bound~\eqref{eq:main-cond} on $\kappa_S(\bs P_0 \bs
A_{\varepsilon,h})$ holds uniformly in $h \le h_0$ and, by
Corollary~\ref{cor:bochner}, uniformly in the stochastic approximation
dimension as well. That is the behaviour exhibited by
Table~\ref{tab:href}, by Table~\ref{tab:sgexp}, and directly by
Figure~\ref{fig:linearity}.

Three limits bound that statement. First, $\kappa_S$ does not by itself
control Krylov convergence for the non-normal matrices at hand
(Section~\ref{sec:oppg}); the convergence claim is carried
by~\eqref{eq:main-gmres}, which is conditional. Second, that estimate rests on
a coercivity-type hypothesis we do not verify for Maxwell, though
Remark~\ref{rmk:fov} checks it numerically at each discretization used here. Third, Remark~\ref{rmk:variationalcrime} identifies the
gap between the exact-surface operators of the theorem and the
piecewise-affine matrices actually assembled. 
The coercivity-type
hypothesis~\cite[Assumption~4]{EscapilInchauspeJerezHanckes2021} asks that the
$X_h$-field of values of $\bs P_0\bs A_{\varepsilon,h}$ and of its inverse
stay away from the origin. We do not establish it for the Maxwell problem, so
\eqref{eq:main-gmres} is asserted conditionally, and it bounds the
\emph{weighted} residual factor, the Euclidean one carrying an extra
$K_{\Lambda_h}$~\cite[eq.~(77)]{EscapilInchauspeJerezHanckes2021}. The
iteration counts of Section~\ref{sec:numerics} are empirical evidence, not a
verification of either bound.

\begin{remark}[Numerical check on the coercivity-type hypothesis]
  \label{rmk:fov}
  Unproven, the hypothesis is nonetheless checkable at each discretization.
  Since the field of values is convex, $\bs 0 \notin W(\bs M)$ holds exactly
  when some rotation makes the Hermitian part positive definite, that is, when
  the \emph{rotated margin}
  \begin{equation}
    \label{eq:fovmargin}
    m(\bs M) := \max_{\theta\in[0,2\pi)} \lambda_{\min}
      \Bigl(\tfrac12\bigl(e^{\iota\theta}\bs M
      + (e^{\iota\theta}\bs M)^{H}\bigr)\Bigr),
    \quad \bs M := \bs P_0\bs A_{\varepsilon,h}\big/
      \|\bs P_0\bs A_{\varepsilon,h}\|_F,
  \end{equation}
  is strictly positive; the rotation matters, the value at $\theta=0$ being
  $5.0\times10^{-4}$ against $2.8\times10^{-3}$ after rotation on the sphere at
  $\kappa=4$. In every smooth case used here $m>0$, certified by a Cholesky
  factorization of the Hermitian part at the exhibited $\theta$ (sphere,
  $\kappa=2$, $4\,821$ dof: $m = 1.19\times10^{-2}$). On the Fichera geometry
  $m<0$, certified in the opposite direction by Rayleigh quotients whose convex
  hull contains the origin, and~\eqref{eq:main-gmres} does not apply. The sign
  changes at a definite electrical size, $m>0$ for $\kappa a \le 1.45$ and
  $m<0$ for $\kappa a \ge 1.50$ with $a$ the cube's half-width, and appears
  insensitive to mesh refinement, the sign at
  $\kappa a = 1.50$ staying negative at three resolutions between $14$ and
  $28$ elements per wavelength. Section~\ref{sec:fichera} sits just above that
  threshold, at $\kappa a = 2$, and already refrains from invoking
  Theorem~\ref{thm:main} there. What fails is not
  Assumption~\ref{ass:stability}: $\mathrm{cond}(\bs N_{XY})$ lies between
  $2.84$ and $2.99$ on all three geometries. The sign change is instead
  consistent with the failure of $\calK$ to be compact on Lipschitz surfaces
  (Section~\ref{sec:calderon}), and supports working in the
  operator-preconditioning framework throughout. We claim no $h$-uniformity
  for $m$: the $h$-uniform statement of Theorem~\ref{thm:main}
  is~\eqref{eq:main-cond}, which needs no coercivity hypothesis.
\end{remark}

\begin{remark}[Regularity]
  \label{rmk:regularity}
  Theorem~\ref{thm:main} requires only the operator
  estimate~\eqref{eq:geocont}, not shape differentiability and not
  holomorphy. The logical chain behind Lemma~\ref{lem:shapepert} is
  Maxwell EFIE shape holomorphy~\cite{EscapilInchauspeJerezHanckes2026}
implies operator-norm Lipschitz bound~\eqref{eq:geocont}
implies $(h,\nu)$-perturbation (Corollary~\ref{cor:hnu})
implies frozen-Calder\'on robustness (Theorem~\ref{thm:main});
  only its first implication is imported from outside this paper.
That import is not frequency-explicit: every constant
  in~\cite{EscapilInchauspeJerezHanckes2026} depends on $\kappa$ and that
  dependence is not tracked there, so~\eqref{eq:geocont}, and with it
  Theorem~\ref{thm:main}, is a statement at fixed $\kappa$ --- the theoretical
  counterpart of the wavenumber behaviour of
  Section~\ref{sec:kappamismatch}, for which we claim no bound. For
  related theory see the classical electromagnetic shape-differentiability
  results of Costabel and Le
  Lou\"er~\cite{CostabelLeLouer2012a,CostabelLeLouer2012b} and the scalar
  boundary-integral shape holomorphy of D\"olz and
  Henr\'iquez~\cite{DoelzHenriquez2024}. Boundaries that are merely Lipschitz, polyhedra in
  particular, remain outside the verified range of
  Assumption~\ref{ass:geocont}; Section~\ref{sec:fichera} probes that gap
  numerically.
\end{remark}

\begin{remark}[Geometric variational crime]
  \label{rmk:variationalcrime}
  Corollary~\ref{cor:hnu}, and hence Theorem~\ref{thm:main}, applies
  \emph{literally} to conforming discretizations on the exact
  surfaces: the matrices it speaks about are those of the
  $X$-conforming RWG/BC spaces carried by $\Gnom$ and $\Gpert$
  themselves. Our implementation instead involves two polyhedral
  approximation layers: the smooth nominal surface $\Gnom$ is replaced
  by the straight-sided mesh surface $\Gamma_{0,h}$, and the perturbed
  surface $\Gpert$ by the mesh $\Gamma_{\bs y,h}$ obtained by moving
  the \emph{vertices} by $\bs\varphi_{\bs y}$ and re-drawing straight
  triangles (i.e.\ deforming by the piecewise-affine interpolant of
  $\bs\varphi_{\bs y}$). Transferring the continuous estimate to the
  matrices actually assembled would require, in addition, a surface
  lift $\ell_h : \Gamma_{0,h} \to \Gnom$ with associated Piola lifts
  of the RWG/BC spaces, together with consistency estimates for the
  lifted bilinear forms. That is the standard, but nontrivial, lifting
  analysis of geometric variational
  crimes~\cite{SaufferNedelec2011}. We do not develop this lifting
  analysis here and therefore do not claim that the assembled matrices
  satisfy Theorem~\ref{thm:main} verbatim.
  Section~\ref{sec:refinement} is accordingly evidence
  that the robustness predicted on the exact surfaces persists for the
  straight-sided implementation, the frozen preconditioner's iteration
  count being constant in $h$ at fixed perturbation amplitude, rather than a
  verification of Theorem~\ref{thm:main} for the assembled polyhedral
  matrices.
\end{remark}

\begin{remark}[Contrast with wavenumber reuse]
  \label{rmk:kappa-mismatch}
 The preceding argument relates to shape perturbations at a \emph{fixed}
  wavenumber: Assumption~\ref{ass:geocont} makes
  $\widetilde\calT_\varepsilon$ a small perturbation of a fixed operator.
  Changing $\kappa$ instead alters the relative scaling of the EFIE's vector-
  and scalar-potential components, the $-\iota\kappa$ and
  $1/(\iota\kappa)$ weights in~\eqref{eq:efie}, and the present analysis
  yields no $(h,\nu)$-bound in the sense of
  Definition~\ref{def:perturbation} uniform over large frequency ratios. We
  therefore make no corresponding theoretical claim for wavenumber reuse;
  Section~\ref{sec:kappamismatch} examines this case numerically, as a
  contrast to the shape-parametric setting.
\end{remark}

\subsection{Bochner-space and stochastic-Galerkin consequence}
\label{sec:bochner}
The uniformity with respect to $\bs y \in U$ in
Assumption~\ref{ass:geocont} has a direct consequence for uncertainty
quantification, where the parametric solution is sought in Bochner
spaces over the parameter
domain~\cite{AylwinJerezHanckesSchwabZech2020} and the conditioning of the
resulting stochastic systems is a known
bottleneck~\cite{EscapilInchauspeJerezHanckes2024,PowellElman2009}. In
this
subsection we write $\mathsf a_{\varepsilon,\bs y}$ for the perturbed
form $\mathsf a_{\bs y}$ of Section~\ref{sec:shapebound}, making the
amplitude explicit, and $\bs A_h(\varepsilon,\bs y)$ for its
DOF-aligned Galerkin matrix. Equip $U$ with a probability measure\footnote{The UQ interpretation of Section~\ref{sec:shapeperts} is
the product uniform measure.} $\rho$, and define the Bochner spaces and
forms
\begin{equation}
  \mathbb X := L^2_\rho(U;X), \quad \mathbb Y := L^2_\rho(U;Y),
  \qquad
  \begin{aligned}
  \mathsf A_\varepsilon(u,v) &:= \int_U \mathsf a_{\varepsilon,\bs y}
    \big(u(\bs y), v(\bs y)\big)\, d\rho(\bs y),\\
  \mathsf A_0(u,v) &:= \int_U \mathsf a_0
    \big(u(\bs y), v(\bs y)\big)\, d\rho(\bs y).
  \end{aligned}
\end{equation}
Corollary~\ref{cor:hnu} and Cauchy--Schwarz in $L^2_\rho(U)$ give
\begin{equation}
  \label{eq:bochner-form}
  \big| \mathsf A_\varepsilon(u,v) - \mathsf A_0(u,v) \big|
  \;\le\; C_{\mathrm{geo}}\,\varepsilon\,
  \|u\|_{\mathbb X}\, \|v\|_{\mathbb Y}.
\end{equation}

For the intrusive (stochastic Galerkin) discretization, let $S_p
\subset L^2_\rho(U)$ be any finite-dimensional stochastic
approximation space, a polynomial-chaos space for instance, and
tensorize \emph{all four} spaces of the OP-PG framework
(Section~\ref{sec:oppg}): alongside $\mathbb X_{h,p} := S_p \otimes
X_h$ and $\mathbb Y_{h,p} := S_p \otimes Y_h$, set $\mathbb V :=
L^2_\rho(U;V)$, $\mathbb W := L^2_\rho(U;W)$, $\mathbb V_{h,p} := S_p
\otimes V_h$, $\mathbb W_{h,p} := S_p \otimes W_h$, and define the
nominal Bochner forms $\mathsf M_0, \mathsf N_0, \mathsf C_0$ from
$\mathsf m, \mathsf n, \mathsf c$ by integration against $\rho$,
exactly as $\mathsf A_0$. In an $L^2_\rho$-\emph{orthonormal} basis
$\{\psi_\alpha\}_{\alpha \in \Lambda_p}$ of $S_p$, the stochastic
Galerkin matrix of $\mathsf A_\varepsilon$ has blocks $[\mathbb
A_{\varepsilon,h,p}]_{\alpha \beta} = \int_U \psi_\alpha(\bs y)
\psi_\beta(\bs y)\, \bs A_h(\varepsilon,\bs y)\, d\rho(\bs y)$,
coupled in the stochastic indices in general. The nominal form and
the pairings, being independent of $\bs y$, lift as tensor products
with the identity, $\mathbb A_{0,h,p} = I_{|\Lambda_p|} \otimes \bs
A_{0,h}$, $\mathbb M_{0,h,p} = I \otimes \bs M_{0,h}$, $\mathbb
N_{0,h,p} = I \otimes \bs N_{0,h}$, $\mathbb C_{0,h,p} = I \otimes
\bs C_{0,h}$, so that
\begin{equation}
  \label{eq:tensorP}
  \mathbb P_{0,h,p} \;=\; I_{|\Lambda_p|} \otimes \bs P_{0,h}:
\end{equation}
the stochastic lift of the frozen nominal Calder\'on preconditioner
is block-diagonal, with repeated copies of the same physical-space
matrix. A single assembly of $\bs T_{YY,0}$ and $\bs N_{XY,0}$ therefore
serves every stochastic mode of the coupled system, so the dominant
dual-mesh cost is paid once for the entire stochastic problem, independently
of $\dim S_p$; Section~\ref{sec:sgexp} measures this. For a non-orthonormal basis the matrix representations
contain the stochastic Gram matrix $G_p$, as $G_p \otimes {}\cdot{}$;
any resulting $\kappa(G_p)$ dependence is a coordinate-basis effect
and disappears after stochastic orthonormalization.

\begin{corollary}[Bochner-space and stochastic-Galerkin robustness]
  \label{cor:bochner}
  Let Assumptions~\ref{ass:stability} and~\ref{ass:geocont} hold,
  and let $0 \le \varepsilon \le \varepsilon_0$ with $\varepsilon <
  \underline\gamma_{\mathsf a_0}/C_{\mathrm{geo}}$.
  \begin{itemize}
  \item[(i)] \textup{(Bochner perturbation estimate)} The
    estimate~\eqref{eq:bochner-form} holds on $\mathbb X \times
    \mathbb Y$.
  \item[(ii)] \textup{(Stochastic-Galerkin consequence)} For every
    $0 < h \le h_0$ and every finite-dimensional $S_p \subset
    L^2_\rho(U)$, the restriction of $\mathsf A_\varepsilon$ to
    $\mathbb X_{h,p} \times \mathbb Y_{h,p}$ satisfies
    Definition~\ref{def:perturbation} relative to $\mathsf A_0$:
    \begin{equation}
      \gamma_{\mathsf A_0}^{-1}\,
      \big| \mathsf A_\varepsilon(u_{h,p},v_{h,p})
        - \mathsf A_0(u_{h,p},v_{h,p}) \big|
      \;\le\; \bar\nu(\varepsilon)\,
      \|u_{h,p}\|_{\mathbb X}\, \|v_{h,p}\|_{\mathbb Y},
    \end{equation}
    with the same $\bar\nu(\varepsilon) =
    C_{\mathrm{geo}}\varepsilon/\underline\gamma_{\mathsf a_0}$
    of~\eqref{eq:nubar-def}, independent of both $h$ and
    $\dim S_p$, and
    \begin{equation}
      \label{eq:bochner-cond}
      \kappa_S\big( \mathbb P_{0,h,p}\, \mathbb A_{\varepsilon,h,p}
      \big) \;\le\; \bar K_\star \cdot
        \frac{1+\bar\nu(\varepsilon)}{1-\bar\nu(\varepsilon)},
    \end{equation}
    with right-hand side uniform in both the spatial and the
    stochastic discretization. If the stochastic-Galerkin
    preconditioned form additionally satisfies the analogue of the
    coercivity-type
    hypothesis~\cite[Assumption~4]{EscapilInchauspeJerezHanckes2021}, with a
    constant uniform in $p$, then the GMRES
    estimate~\eqref{eq:main-gmres} follows with the same
    $h,p$-uniform constants.
  \end{itemize}
  In particular, increasing the stochastic approximation order does
  not deteriorate the bound for the frozen nominal Calder\'on
  preconditioner.
\end{corollary}

\begin{proof}
  Part (i) is Corollary~\ref{cor:hnu} integrated against $\rho$, followed
  by Cauchy--Schwarz in $L^2_\rho(U)$. For (ii) the form perturbation is the
  restriction of~\eqref{eq:bochner-form}, so all that is required is that the
  nominal constants on the tensor-product spaces equal their deterministic
  counterparts. Expanding in the orthonormal basis of $S_p$ gives $\mathsf
  A_0(u,v) = \sum_\alpha \mathsf a_0(u_\alpha,v_\alpha)$ and $\|u\|_{\mathbb
  X}^2 = \sum_\alpha \|u_\alpha\|_X^2$, so $\|\mathsf A_0\| \le \|\mathsf
  a_0\|$ by Cauchy--Schwarz in $\alpha$, with equality on testing against
  rank-one tensors $\psi \otimes u_0$. For the inf-sup constant, note first
  that Assumption~\ref{ass:stability} makes the induced operator $\bs A_{0,h}
  : X_h \to Y_h'$ injective, hence bijective since $\dim X_h = \dim Y_h$, so
  that $\gamma_{\mathsf a_0,h} = \|\bs A_{0,h}^{-1}\|^{-1}$ in the operator
  norms induced by the discrete $X$- and $Y'$-norms. The operator induced by
  $\mathsf A_0$ on the tensor-product spaces is $I_{S_p} \otimes \bs A_{0,h}$,
  whose inverse is $I_{S_p} \otimes \bs A_{0,h}^{-1}$; the Hilbertian tensor
  norm is a cross norm, so $\|I_{S_p} \otimes \bs A_{0,h}^{-1}\| = \|\bs
  A_{0,h}^{-1}\|$ and therefore $\gamma_{\mathsf A_0} = \gamma_{\mathsf
  a_0,h}$ exactly, independently of $\dim S_p$. The same argument applied to $\mathsf M_0, \mathsf N_0, \mathsf
  C_0$ leaves every constant entering $K_\star$ unchanged, uniformly in $\dim
  S_p$, and identifies $\mathbb P_{0,h,p} = I \otimes \bs P_{0,h}$ as the
  OP-PG preconditioner of the tensorized problem, repeated deterministic
  blocks at the matrix level (cf.~\eqref{eq:tensorP}). Dividing the form
  perturbation by $\gamma_{\mathsf A_0} \ge \underline\gamma_{\mathsf a_0}$
  gives $\bar\nu(\varepsilon)$, and Theorems~\ref{thm:ejh-op}
  and~\ref{thm:ejh-gmres} at $\mu=0$, with constants enlarged as
  in~\eqref{eq:main-cond}, give~\eqref{eq:bochner-cond} and the conditional
  GMRES statement.
\end{proof}

Corollary~\ref{cor:bochner} applies to standard stochastic Galerkin
discretizations of the pulled-back parametric EFIE. It is analogous
to mean-based and Kronecker-product preconditioning for stochastic
Galerkin systems~\cite{PowellElman2009}, but the repeated
deterministic block here is the \emph{nominal-geometry} Calder\'on
preconditioner rather than the preconditioner associated with the
mean operator. In general $\bs A_{0,h} \ne \mathbb E[\bs
A_h(\varepsilon,\cdot)]$ for a nonlinear geometric perturbation. Its
extension to the tensorized moment systems arising in first-order
sparse BEM~\cite{EscapilInchauspeJerezHanckes2024}, where a natural
candidate preconditioner is of the form $\bs P_0 \otimes
\overline{\bs P_0}$ rather than $I \otimes \bs P_0$, is left for
future work. Section~\ref{sec:sgexp} tests the
stochastic-dimension-uniform prediction numerically.

\subsection{Implementation: DOF alignment via mesh connectivity}
\label{sec:implementation}
Although the physical RWG and BC basis functions depend on the vertex
geometry, their indexing and support incidence structure are
determined entirely by mesh connectivity. Both families are indexed by
the \emph{interior edges of the primal mesh}: an RWG function is
attached to an edge via the two primal triangles sharing it, and its
dual BC function is attached to the \emph{same} primal edge, though
represented as a linear combination of RWG-type functions on the
barycentric refinement, supported on the union of barycentric cells
surrounding the edge's two endpoints~\cite{BuffaChristiansen2007,
AndriulliEtAl2008}. Because our perturbation map~\eqref{eq:pertmap}
preserves mesh connectivity exactly (Section~\ref{sec:shapeperts}),
the permutation $\Pi_X$ (resp.\ $\Pi_Y$) identifying $X_h(\Gnom)$ with
$X_h(\Gpert)$ (resp.\ $Y_h$) can be constructed purely
combinatorially, without any geometric computation. We fingerprint
each nominal RWG function by the (sorted) pair of \emph{vertex
indices} of the primal edge it lives on, and each perturbed RWG
function likewise, then match fingerprints. For BC functions our
implementation uses an equivalent connectivity-derived signature: the
sorted set of \emph{original} primal-mesh vertices, excluding
barycentric-refinement-introduced midpoints and centroids, contained
in each BC function's support (the vertex patches of the associated
edge's endpoints), which identifies the underlying primal edge
uniquely; bijectivity of the resulting matching is verified at
runtime for every perturbed mesh. This permutation-based alignment is
exact (not approximate) whenever connectivity is preserved, which is
guaranteed by construction for the vertex-displacement perturbation
family~\eqref{eq:pertmap} used throughout this paper, and is the
implementation realized in Section~\ref{sec:numerics}.

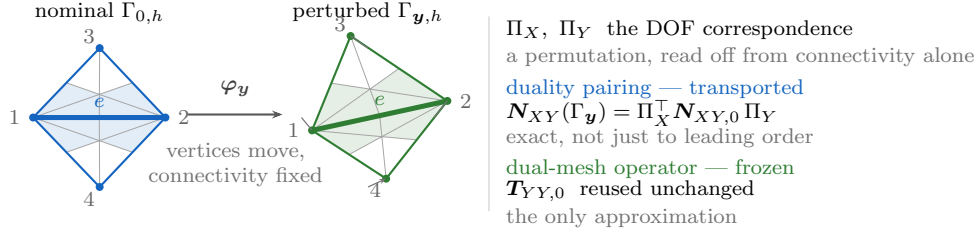
\begin{figure}[t]
  \centering
  \begin{tikzpicture}[x=1cm,y=1cm,font=\scriptsize]
      \coordinate (a1) at (0,0);      \coordinate (a2) at (1.75,0);
      \coordinate (a3) at (0.88,0.92); \coordinate (a4) at (0.88,-0.92);
      \meshpatch{nomblue}{a1}{a2}{a3}{a4}
      \node[nomblue,above=0pt] at ($ (a1)!0.5!(a2) $) {$e$};
      \node[black!60,left=1pt]       at (a1) {1};
      \node[black!60,right=1pt]      at (a2) {2};
      \node[black!60,above left=-2pt] at (a3) {3};
      \node[black!60,below left=-2pt] at (a4) {4};
      \node at (0.88,1.38) {nominal $\Gamma_{0,h}$};
  
      \begin{scope}[xshift=3.55cm]
        \coordinate (b1) at (0.14,-0.17);  \coordinate (b2) at (1.93,0.22);
        \coordinate (b3) at (0.65,1.08);   \coordinate (b4) at (1.12,-0.81);
        \coordinate (c1) at (0,0);         \coordinate (c2) at (1.75,0);
        \coordinate (c3) at (0.88,0.92);   \coordinate (c4) at (0.88,-0.92);
        \meshpatch{pergreen}{b1}{b2}{b3}{b4}
        \draw[-{Stealth[length=2.6pt]},black!55,line width=0.45pt]
          (c1) -- (b1) (c2) -- (b2) (c3) -- (b3) (c4) -- (b4);
        \node[pergreen,above=0pt] at ($ (b1)!0.5!(b2) $) {$e$};
        \node[black!60,left=1pt]        at (b1) {1};
        \node[black!60,right=1pt]       at (b2) {2};
        \node[black!60,above left=-2pt] at (b3) {3};
        \node[black!60,below left=-2pt] at (b4) {4};
        \node at (0.88,1.38) {perturbed $\Gamma_{\bs y,h}$};
      \end{scope}
  
      \draw[-{Stealth[length=4pt]},black!65,line width=0.8pt]
        (2.05,0.04) -- (3.30,0.04);
      \node[black!75,above=1pt] at (2.68,0.12) {$\bs\varphi_{\bs y}$};
      \node[black!55,align=center,below=1pt] at (2.68,-0.16)
        {vertices move,\\connectivity fixed};
  
      \begin{scope}[xshift=6.25cm]
        \draw[black!20,line width=0.4pt] (-0.22,-1.24) -- (-0.22,1.34);
        \node[anchor=north west,inner sep=0pt] at (0,1.30)
          {$\Pi_X,\ \Pi_Y$ \;the DOF correspondence};
        \node[anchor=north west,black!55,inner sep=0pt] at (0,0.94)
          {a permutation, read off from connectivity alone};
        \node[anchor=north west,nomblue,inner sep=0pt] at (0,0.52)
          {duality pairing --- transported};
        \node[anchor=north west,inner sep=0pt] at (0,0.24)
          {$\bs N_{XY}(\Gamma_{\bs y})=\Pi_X^{\top}\bs N_{XY,0}\,\Pi_Y$};
        \node[anchor=north west,black!55,inner sep=0pt] at (0,-0.14)
          {exact, not just to leading order};
        \node[anchor=north west,pergreen,inner sep=0pt] at (0,-0.54)
          {dual-mesh operator --- frozen};
        \node[anchor=north west,inner sep=0pt] at (0,-0.82)
          {$\bs T_{YY,0}$ \;reused unchanged};
        \node[anchor=north west,black!55,inner sep=0pt] at (0,-1.20)
          {the only approximation};
      \end{scope}
    \end{tikzpicture}
  \caption{Transport of the Calder\'on preconditioner. Shaded: the
    barycentric patches supporting the BC function attached to the interior
    edge $e$. The pairing transport is exact for any amplitude
    (Lemma~\ref{lem:nxy-exact}).}
  \label{fig:transport}
\end{figure} 
The duality pairing is carried \emph{exactly} by this transport, and not
merely to leading order in $\varepsilon$.

\begin{lemma}[exactness of the transported pairing]
  \label{lem:nxy-exact}
  Let $\tau_0, \tau_r \subset \R^3$ be flat triangles with unit normals $\bs
  n_0, \bs n_r$, and let $\Phi : \tau_0 \to \tau_r$ be affine and bijective.
  Write $\bs F := \mathrm D_\Gamma\Phi : T\tau_0 \to T\tau_r$ for its
  (constant) tangential differential, a linear isomorphism of the two
  tangent planes, and $J := \mathrm{d}s_r/\mathrm{d}s_0$ for the constant
  surface Jacobian, so that $|\det\bs F| = J$ in orthonormal bases of
  $T\tau_0$ and $T\tau_r$. Orient $\bs n_r$ so that $\Phi$ preserves the
  induced orientation, i.e.\ $\det\bs F = J > 0$. Let $\calP$ be the
  contravariant surface Piola transform, $(\calP\bs u)\circ\Phi = J^{-1}\bs
  F\bs u$. Then
  \begin{equation}
    \label{eq:pairing-exact}
    \big\langle \calP\bs u, \calP\bs v \big\rangle_{\times,\tau_r}
    \;=\; \big\langle \bs u, \bs v \big\rangle_{\times,\tau_0}
  \end{equation}
  for all tangential fields $\bs u,\bs v$ on $\tau_0$. Consequently, if the
  nominal and perturbed triangulations share their connectivity and
  correspond element by element under such a map, and if the
  Buffa--Christiansen coefficients are determined by that connectivity
  alone, then
  \begin{equation}
    \label{eq:nxy-exact}
    \bs N_{XY}(\Gpert) \;=\; \Pi_X^{\top}\,\bs N_{XY,0}\,\Pi_Y
  \end{equation}
  identically, with no dependence on the vertex positions.
\end{lemma}

\begin{proof}
  For tangential fields on a flat triangle the integrand of $\langle
  \cdot,\cdot\rangle_\times$ is minus the oriented area form: $(\bs a\times\bs n)\cdot\bs b = -(\bs a\times\bs b)\cdot\bs n$ and, writing
  $\det_{\tau}(\bs a,\bs b)$ for the determinant of the coefficients of
  tangential $\bs a,\bs b$ in a positively oriented orthonormal basis of
  $T\tau$, one has $(\bs a\times\bs b)\cdot\bs n = \det_{\tau}(\bs a,\bs b)$.
  Since $\bs F$ is linear from $T\tau_0$ to $T\tau_r$,
  \begin{equation*}
    \det\nolimits_{\tau_r}(\bs F\bs u, \bs F\bs v)
      \;=\; (\det \bs F)\,\det\nolimits_{\tau_0}(\bs u,\bs v)
      \;=\; J\,\big(\bs u \times \bs v\big)\cdot\bs n_0 .
  \end{equation*}
  Applying this to $\calP\bs u = J^{-1}\bs F\bs u$ and $\calP\bs v =
  J^{-1}\bs F\bs v$ gives $(\calP\bs u \times \calP\bs v)\cdot\bs n_r =
  J^{-2}\cdot J\,(\bs u \times \bs v)\cdot\bs n_0 = J^{-1}(\bs u \times \bs
  v)\cdot\bs n_0$, and substituting $\mathrm{d}s_r = J\,\mathrm{d}s_0$ in the
  integral over $\tau_r$ cancels the remaining factor. Four Jacobian factors
  cancel in pairs: one $J^{-1}$ from each of the two Piola transforms,
  against one $J$ from the tangential determinant and one $J$ from the
  change of measure.

  For~\eqref{eq:nxy-exact}, the perturbation~\eqref{eq:pertmap} displaces
  vertices and leaves triangles flat, so on each element it restricts to an
  affine bijection of the above form; for $\varepsilon$ small enough that no
  element degenerates, $\bs F$ is a small perturbation of the identity on the
  tangent plane and the orientation convention above is the one inherited
  from $\Gnom$. The RWG functions on $\Gpert$ are then the Piola
  transports of their nominal counterparts, being given by the same
  reference-element formula; the BC functions are, under the stated
  hypothesis, the same fixed combinations of RWG-type functions on the
  barycentric refinement, and so transport with them. Applying
  \eqref{eq:pairing-exact} elementwise and summing over the element
  correspondence gives the claim; $\Pi_X$ and $\Pi_Y$ record the
  relabelling alone.
\end{proof}

Lemma~\ref{lem:nxy-exact} is what separates the transport used here from an
update formula: it is an identity, not an approximation that happens to be
accurate. Both hypotheses are met by the discretization used here, and
neither is vacuous. The RWG shape functions of \texttt{BEAST.jl} carry the
Jacobian normalisation and no edge-length factor, so they are Piola-covariant
by construction; an edge-length-normalised convention would transport them
only up to the diagonal factor $\ell_e^{(0)}/\ell_e^{(\bs y)}$. The
Buffa--Christiansen weights are the combinatorial $1/n$ of the surrounding
patch in the default construction, but face areas and port lengths in the
area-scaled variant, for which the lemma fails. The transported nominal pairing
agrees with the freshly assembled one to the level of the floating-point
assembly itself at a $30\%$ perturbation, while the geometry-dependent $\bs
T_{YY}$ does not (Section~\ref{sec:campaign}).

\section{A practical cost model}
\label{sec:numanalysis}

Let $t_Z, t_T, t_N, t_b$ denote the assembly times of $\bs Z_{XX}, \bs
T_{YY}, \bs N_{XY}, \bs b$ respectively, and let $t_{\mathrm{solve}}$
denote the GMRES solve time (dominated by $O(\mathrm{dof}^2)$
dense mat-vecs per iteration). A fresh solve of a new realization costs
\begin{equation}
  \label{eq:cost-fresh}
  T_{\mathrm{fresh}} = t_Z + t_T + t_N + t_b + t_{\mathrm{solve}}
    (\mathrm{iters}_{\mathrm{fresh}}),
\end{equation}
The frozen strategy assembles $\bs T_{YY,0}, \bs N_{XY,0}$
\emph{once} and reuses them, via the permutation $\Pi_X,\Pi_Y$ of
Section~\ref{sec:implementation}, for every subsequent realization. It
costs, per new realization,
\begin{equation}
  \label{eq:cost-frozen}
  T_{\mathrm{frozen}} = t_Z + t_b + t_{\mathrm{solve}}
    (\mathrm{iters}_{\mathrm{frozen}}).
\end{equation}
Here $t_{\mathrm{solve}}$ is written as a single function of the
iteration count for both strategies, i.e.\ the per-iteration cost is taken
to be identical. 
In the runs reported here it is not, by a factor of about $3.6$, for a
reason of storage rather than of conditioning that
Section~\ref{sec:campaign} measures and explains; applying the permutation
lazily inside the matrix--vector products removes it, so the speedups
reported in Section~\ref{sec:numerics} are correspondingly conservative. The alignment itself is a permutation of row and column indices, built
combinatorially in $O(\mathrm{dof})$ time, so its time contribution
$t_{\mathrm{align}}$ is negligible against every term
in~\eqref{eq:cost-frozen} and is omitted; its memory cost, the
$O(\mathrm{dof}^2)$ permuted copy coexisting with the nominal blocks, is
not.
Since $t_T \gg t_Z, t_b$
(Table~\ref{tab:componenttiming}) and
$\mathrm{iters}_{\mathrm{frozen}} \approx \mathrm{iters}_{\mathrm{fresh}}$
(Theorem~\ref{thm:main}; observed throughout
Section~\ref{sec:numerics}), the frozen strategy's per-realization
speedup $T_{\mathrm{fresh}}/T_{\mathrm{frozen}}$ is governed almost
entirely by the ratio $(t_Z+t_T+t_N+t_b)/(t_Z+t_b)$, i.e.\ by how much
of the total assembly cost is attributable to the dual-mesh operator
$\bs T_{YY}$. For an $N$-realization many-query campaign (e.g.\ Monte
Carlo shape UQ), the total cost ratio is
\begin{equation}
  \label{eq:speedup}
  \begin{aligned}
  \frac{T_{\mathrm{fresh}}^{\mathrm{total}}}{T_{\mathrm{frozen}}^{
    \mathrm{total}}} &= \frac{t_Z+t_T+t_N+t_b + t_{\mathrm{solve}}(
    \mathrm{iters}_{\mathrm{fresh}})}{t_Z + \tfrac1N(t_T+t_N) + t_b +
    t_{\mathrm{solve}}(\mathrm{iters}_{\mathrm{frozen}})} \\[2pt]
    &\xrightarrow[N\to\infty]{}\; \frac{t_Z+t_T+t_N+t_b +
    t_{\mathrm{solve}}(\mathrm{iters}_{\mathrm{fresh}})}{t_Z + t_b +
    t_{\mathrm{solve}}(\mathrm{iters}_{\mathrm{frozen}})}.
  \end{aligned}
\end{equation}
The denominator amortizes the one-time nominal assembly $t_T + t_N$
over the $N$ frozen realizations; the fresh strategy incurs these
costs at every realization, which is why the numerator contains the
full $t_T + t_N$ per query. Section~\ref{sec:numerics} reports this
ratio directly from measured timings.

Three figures follow from~\eqref{eq:speedup} and should not be confused.
The \emph{measured} factor is $T_{\mathrm{fresh}}/T_{\mathrm{frozen}}$ per
new realization, $29$--$31\times$ in the dense runs of
Section~\ref{sec:campaign}; since the per-realization frozen cost
in~\eqref{eq:cost-frozen} already excludes $t_T+t_N$, this is at the same
time the $N \to \infty$ limit of~\eqref{eq:speedup}, and it is the figure
quoted in the abstract and in Section~\ref{sec:conclusions}. Removing the
storage artifact of Section~\ref{sec:campaign} raises it to $\approx
36\times$. Neglecting the solve cost altogether gives $1 +
(t_T+t_N)/(t_Z+t_b) \approx 38$, the ceiling the other two approach.

What governs all three is the ratio $t_T/t_Z$, equal to $37$ for the dense
assembly of Table~\ref{tab:componenttiming}. Fast-assembly implementations
reduce it, and do so by accelerating the preconditioner rather than the
operator, a preconditioner tolerating far cruder approximation than the
system it preconditions. Applying~\eqref{eq:speedup} to the published
$\mathcal H$-matrix timings
of~\cite{EscapilInchauspeJerezHanckes2019,FierroJerezHanckes2020} gives
$2.0$ to $2.3$: under compression the solve is no longer negligible against
assembly, and~\eqref{eq:speedup} amortizes $t_T+t_N$ against whatever
per-query cost remains, so compression shrinks the numerator without
shrinking that remainder. Neglecting the solve instead would give $4.2$ to
$5.9$; we quote the former.

\begin{table}[t]
  \centering
  \caption{Per-matrix assembly time on the unit sphere, $\kappa=2$,
    $h=0.1$ (dof $=4827$), realization~1. The dual-mesh operator $\bs
    T_{YY}$ dominates assembly cost regardless of perturbation
    amplitude.}
  \label{tab:componenttiming}
  \begin{tabular}{@{}lrrrrr@{}}
    \toprule
    Amplitude & $t_{Z_{XX}}$ (s) & $t_{T_{YY}}$ (s) & $t_{N_{XY}}$ (s) &
      $t_b$ (s) & $t_{T_{YY}}$ / total \\
    \midrule
    3\%  & 11.64 & 430.34 & 0.19 & 0.035 & 97.3\% \\
    30\% & 11.70 & 431.87 & 0.18 & 0.035 & 97.3\% \\
    \bottomrule
  \end{tabular}
\end{table}

\section{Numerical experiments}
\label{sec:numerics}

All experiments discretize~\eqref{eq:efie} with RWG/BC elements via the
BEAST.jl boundary element library. The incident field is the unit-amplitude
plane wave
\begin{equation}
  \bs U^{\mathrm{inc}}(\bs x) = \bs p\, e^{-\iota\kappa\, \bs d\cdot\bs
  x}, \qquad \bs d = \hat{\bs z},\quad \bs p = \hat{\bs x},\quad \bs
  d\cdot\bs p = 0,
\end{equation}
in every experiment. The problem sizes below are moderate, and
deliberately so: what they limit is the absolute timing, not the
conclusion. The transport is exact by connectivity
(Lemma~\ref{lem:nxy-exact}) and the bound of Theorem~\ref{thm:main} is
uniform in $h$, so neither the mechanism nor its analysis refers to the
number of unknowns; the speedups of Section~\ref{sec:numanalysis}, which
do, are reported as such. The outer solver is (flexible, unrestarted)
GMRES~\cite{Saad1993} (flexibility is required since the Calder\'on
preconditioner is applied through inner GMRES solves against $\bs
N_{XY}$ and $\bs N_{XY}^{\top}$), target residual tolerance $10^{-8}$
(both absolute and relative), maximum $1500$ iterations. The inner
solves use unrestarted GMRES with the \emph{same fixed} tolerances
(absolute and relative $10^{-8}$, maximum $1500$ iterations), zero
initial guess, and no further preconditioning; the inner tolerance is
not varied adaptively.
Section~\ref{sec:sgexp} is the one exception, for the reason given
there. Retaining inner GMRES elsewhere is conservative: a factorization would be computed once for the
frozen strategy, which reuses $\bs N_{XY,0}$ at every realization, and afresh
for the freshly assembled one, which does not, so adopting it throughout
would \emph{raise} the speedups below while removing the per-iteration
asymmetry of Section~\ref{sec:numanalysis} and the storage artifact of
Section~\ref{sec:campaign}.
\footnote{The stopping criterion applies to the residual of the system GMRES
iterates on, which for the preconditioned solves is the left-preconditioned
residual $\|\bs P(\bs b-\bs Z\bs u_k)\|/\|\bs P\bs b\|$, whereas the tables
report the unpreconditioned true residual $\|\bs b-\bs Z\bs u\|/\|\bs b\|$ of
the returned solution. Meeting a $10^{-8}$ tolerance in the preconditioned
norm therefore does not give the same numerical value in the unpreconditioned
one, which is why the preconditioned runs below show true residuals of order
$10^{-8}$--$10^{-7}$. All reported true residuals are at or below
$1.5\times10^{-6}$, the largest at the smallest wavenumber of the mismatch
sweep of Section~\ref{sec:kappamismatch}, and elsewhere at or below
$6.3\times10^{-7}$.}
For the numerical experiments we equip the finite-dimensional
truncation of the deterministic parameter domain $U$ of
Section~\ref{sec:shapeperts} with its product uniform measure and draw
random members of the parametric family: shape perturbations of the
sphere (radius $R=1$) displace each nominal vertex radially to
\begin{equation}
  \label{eq:pertformula}
  r_{\bs y}(\theta,\phi)
  \;=\; R\left(1 + \varepsilon \sum_{l=2}^{6} \sum_{m=-l}^{l}
    \frac{y_{lm}}{l^{2}}\, Y_l^m(\theta,\phi)\right),
  \qquad y_{lm} \overset{\text{i.i.d.}}{\sim} \mathrm{Unif}(-1,1),
\end{equation}
where $(\theta,\phi)$ are the vertex's spherical angles and $Y_l^m$
are the real spherical harmonics in the standard convention:
orthonormal on the unit sphere, with the Condon--Shortley phase.
The $1/l^2$ decay makes low-degree modes dominate, keeping perturbed
shapes star-shaped. The coefficients are drawn once per realization
via
 $$\texttt{MersenneTwister(1000+realization)} $$
 
 \noindent
 for reproducibility
(realization~$0$ is always the exact nominal shape), and an amplitude
sweep rescales the \emph{same} draw $\{y_{lm}\}$ by varying
$\varepsilon$ only. The ellipsoid inherits the same deformation field:
a nominal vertex $(a u_1, b u_2, c u_3)$, image of the unit-sphere
vertex $\bs u$, moves to $(1 + \delta(\theta,\phi))(a u_1, b u_2, c
u_3)$, where $\delta$ is the parenthesized sum
in~\eqref{eq:pertformula} evaluated at the angles of $\bs u$ (not of
the ellipsoid point), so the perturbation's angular structure is
identical to the sphere case. We
report a canonical sphere (radius~1) and a triaxial ellipsoid
(semi-axes $(a,b,c) = (1.5,1.0,0.7)$, no residual rotational symmetry
beyond the discrete sign-flip group), the latter perturbed identically
via the underlying unit-sphere angular parametrization
(Section~\ref{sec:ellipsoid}), and a non-convex, non-smooth Fichera
corner geometry under a different, spatially localized perturbation
family (Section~\ref{sec:fichera}).

\emph{Reproducibility.} All computations use Julia v1.11.5 with BEAST.jl
v2.8.0, CompScienceMeshes.jl v0.10.0 and Makeitso.jl v2.1.1, with dense
(uncompressed) Galerkin assembly on an Apple M3 Max (36\,GB RAM; 1 Julia
thread, 10 BLAS threads). The NASA-almond meshes, field coefficients with
their measured $\sup|\mathrm D\bs V|$, amplitude ladder, seeds, timings and
iteration counts are archived with the code, so every perturbed geometry of
Section~\ref{sec:almond} is reproducible exactly and without re-drawing.
Timings should be read relative to this configuration; iteration counts and
residuals are insensitive to it, up to floating-point-library and threading
effects.

\begin{figure}[t]
  \centering
  \includegraphics[width=\textwidth]{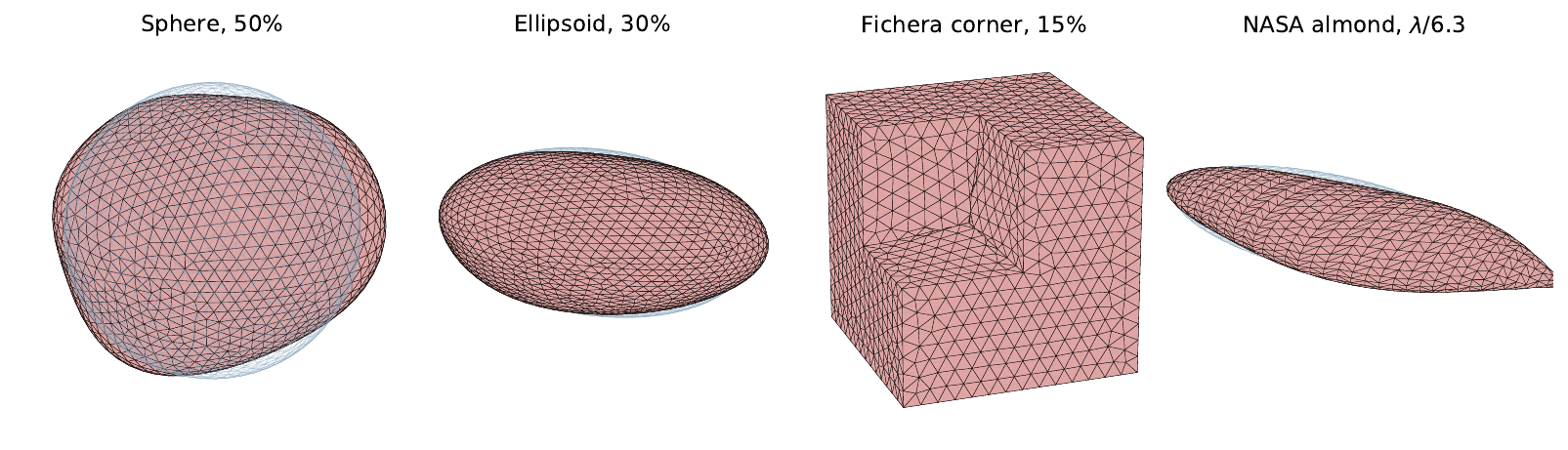}
  \caption{Perturbed meshes, each drawn over a thin outline of its own
    nominal mesh. Sphere and ellipsoid: spherical-harmonic radial deformation
    at the largest amplitudes of Tables~\ref{tab:sphereamp}
    and~\ref{tab:ellipsoid} ($50\%$, $30\%$). Fichera corner: the
    corner-reaching perturbation of Section~\ref{sec:fichera} at
    $\varepsilon=15\%$, the configuration of
    Table~\ref{tab:ficheracorner}, viewed into the reentrant notch. NASA
    almond: one draw of the campaign of Section~\ref{sec:almond} at the top
    amplitude $\varepsilon = \lambda/6.3$ (case~U, seed $5012$), whose
    largest displacement is $13.30$\,mm and mean over the surface
    $9.75$\,mm.}
  \label{fig:mesh-comparison}
\end{figure}

\subsection{Sphere: amplitude sweep}
\label{sec:sphereamp}

Fixing $\kappa=2$, $h=0.1$ (dof$\,{=}\,4827$) and a single realization
(seed 1001), Table~\ref{tab:sphereamp} sweeps the nominal perturbation amplitude
$\varepsilon$ from 3\% to 50\% (actual achieved radial displacement up
to $\pm 17.1\%$ at $\varepsilon=50\%$, that is, up to $1.7h$ and about
$\lambda/18$ at this mesh and wavenumber; all meshes remain valid, minimum
triangle area $> 0$ throughout). Plain (unpreconditioned) GMRES requires
$272$--$290$ iterations across the sweep; the DOF-aligned frozen
preconditioner requires only $11$--$15$ iterations, essentially tracking
the freshly assembled Calder\'on preconditioner's $10$--$13$ iterations,
a gap of at most $2$ iterations even at 50\% amplitude, consistent
with the perturbative robustness predicted by
Theorem~\ref{thm:main}.

\begin{table}[t]
  \centering
  \scriptsize
  \caption{Sphere, amplitude sweep, $\kappa=2$, $h=0.1$, single
    realization (seed 1001). Max.\ achieved radial displacement grows
    with nominal amplitude; frozen and fresh Calder\'on iteration
    counts remain close throughout. The last three columns report the
    converged true relative GMRES residual $\|\bs b-\bs Z\bs u\|/
    \|\bs b\|$ for each strategy, confirming all three solves reach
    comparable accuracy despite the large iteration-count gap for
    plain GMRES.}
  \label{tab:sphereamp}
  \begin{tabular}{@{}rrrrrrrrr@{}}
    \toprule
    & & & & & Frozen & \multicolumn{3}{c}{True residual} \\
    Amplitude & Max.\ disp.\ & Plain it. & Frozen it. & Fresh it.
      & $-$ fresh & plain & frozen & fresh \\
    \midrule
    3\%  & 1.0\%  & 272 & 11 & 10 & 1 & 9.0e-9 & 5.3e-7 & 5.3e-7 \\
    10\% & 3.4\%  & 273 & 11 & 11 & 0 & 9.4e-9 & 5.4e-7 & 4.0e-8 \\
    20\% & 6.9\%  & 275 & 13 & 12 & 1 & 9.3e-9 & 2.2e-7 & 5.2e-8 \\
    30\% & 10.3\% & 281 & 13 & 12 & 1 & 9.9e-9 & 3.6e-8 & 5.9e-8 \\
    40\% & 13.7\% & 287 & 14 & 13 & 1 & 9.1e-9 & 3.7e-8 & 6.8e-8 \\
    50\% & 17.1\% & 290 & 15 & 13 & 2 & 9.6e-9 & 2.8e-8 & 7.4e-8 \\
    \bottomrule
  \end{tabular}
\end{table}

\subsection{Sphere: statistical campaign (\texorpdfstring{$N=20$}{N=20}
  realizations)}
\label{sec:campaign}
To confirm the amplitude sweep is not an artifact of a single random draw,
we ran $N=20$ independent realizations at fixed amplitude, for both
$\varepsilon=3\%$ and $\varepsilon=30\%$ (Table~\ref{tab:campaign}). Frozen
tracks fresh throughout: the mean gap is at most one iteration at either
amplitude and never exceeds two for any individual realization, against
$272$--$282$ plain iterations. All runs converged, and the frozen and fresh
solutions agree with the plain-GMRES solution to within $2\times10^{-7}$
relative $\ell^2$ norm.

\begin{table}[t]
  \centering
  \scriptsize
  \caption{Sphere, $N=20$-realization campaign, $\kappa=2$, $h=0.1$.
    Iteration counts (min--max, mean) and average practical wall-clock
    cost per new realization.}
  \label{tab:campaign}
  \begin{tabular}{@{}lrrrrr@{}}
    \toprule
    Amplitude & Plain iters & Frozen iters & Fresh iters &
      $T_{\mathrm{fresh}}^{\mathrm{avg}}$ (s) &
      $T_{\mathrm{frozen}}^{\mathrm{avg}}$ (s) \\
    \midrule
    3\%  & 272 (272--272) & 10.35 (10--11) & 10.00 (10--10) & 443.1 & 14.5 \\
    30\% & 279.25 (275--282) & 12.85 (12--13) & 11.95 (11--12) & 444.8 & 15.3 \\
    \bottomrule
  \end{tabular}
\end{table}

Table~\ref{tab:campaign}'s wall-clock columns
use~\eqref{eq:cost-fresh}--\eqref{eq:cost-frozen} with the component timings of
Table~\ref{tab:componenttiming} and the campaign-averaged GMRES solve
times ($0.87$--$1.05$\,s fresh, $2.80$--$3.55$\,s frozen). The frozen
strategy's inner GMRES solves against $\bs N_{XY,0}$ cost about $3.6$ times
more per outer iteration than the fresh strategy's, at comparable
\emph{outer} iteration counts. The cause is storage, not conditioning. The
RWG--BC pairing is fixed by mesh connectivity alone, so the DOF-aligned
nominal pairing and the perturbed one agree to machine precision: at
$\varepsilon = 30\%$ we measure $\|\Pi_X^{\top}\bs N_{XY,0}\Pi_Y - \bs
N_{XY}\|_F / \|\bs N_{XY}\|_F = 7.7\times10^{-16}$, against
$3.9\times10^{-2}$ for the corresponding difference in $\bs T_{YY}$, which
does depend on the geometry. What differs is that applying $\Pi_X,\Pi_Y$ to
a matrix that is $0.4\%$ populated densifies it, and $\bs N_{XY}$ enters
through two inner solves per outer iteration. Re-running the frozen
preconditioner with the \emph{same} pairing stored densely reproduces the
frozen cost, $0.209$ against $0.216$\,s per outer iteration, at identical
iteration counts and residuals, while the sparse format gives $0.060$\,s.
The overhead is therefore removed by the lazy permutation of
Section~\ref{sec:numanalysis} rather than by re-assembling $\bs N_{XY}$:
refreshing it at every realization changes no outer iteration count at any
amplitude tested. The same artifact reappears on the NASA almond of
Section~\ref{sec:almond}, at $2.9\times$ rather than $3.6\times$, on a
geometry and a mesh with nothing else in common; that it survives the change
of test case is what one expects of a storage effect and not of a
conditioning one. The resulting speedup per new realization is
$T_{\mathrm{fresh}}/T_{\mathrm{frozen}} = 30.6\times$ at $3\%$ and
$29.1\times$ at $30\%$, essentially independent of amplitude over the range
tested, since it is governed by the assembly-cost ratio of
Table~\ref{tab:componenttiming} rather than by the (comparably small, and
comparably matched) iteration counts.

\subsection{Direct test of the perturbation estimate}
\label{sec:linearity}
The experiments so far report iteration counts, which the theory reaches only
through Corollary~\ref{cor:hnu} and Theorem~\ref{thm:main}. Two predictions of
that chain are measurable directly: the perturbation of the bilinear form grows
linearly in $\varepsilon$ with a constant independent of $h$, and the resulting
bound on $\kappa_S(\bs P_0 \bs A_{\varepsilon,h})$ is likewise $h$-uniform.
Both quantities are computable at these problem sizes, in the DOF-aligned
nominal basis of Section~\ref{sec:implementation}.

Figure~\ref{fig:linearity}(a) plots $\|\bs A_{\varepsilon,h} - \bs
A_{0,h}\|_F$ against $\varepsilon$ at three mesh levels spanning a fourfold
range in degrees of freedom ($1224$, $2292$ and $4827$ at $h = 0.2,
0.15, 0.1$: the meshes of Table~\ref{tab:href}). The least-squares slopes on
the logarithmic scale are indistinguishable from one, with $R^2 >
1-10^{-6}$, so the dependence is linear to within $0.2\%$. For fixed $h$
the expansion $\bs A_{\varepsilon,h} - \bs A_{0,h} = \varepsilon \bs D_h +
O(\varepsilon^2)$, which holds because the amplitude sweep rescales a single
fixed deformation direction, predicts exactly this in any fixed matrix norm.

Panel~(a) therefore tests first-order dependence at fixed $h$ and nothing
more: it does \emph{not} test the $h$-uniform trace-space
estimate~\eqref{eq:geocont}, since the equivalence between coefficient-space
and trace-space norms is itself $h$-dependent --- which is also why the three
curves sit at different heights. Mesh-uniformity is tested separately in
panel~(b), through $\kappa_S(\bs P_0 \bs A_{\varepsilon,h})$ itself, computed
from the eigenvalues of the frozen-preconditioned matrix. The three curves
lie on top of one another, agreeing to within $0.8\%$ at every amplitude
tested, with nominal values $1.798$, $1.796$ and $1.795$; across the sweep
the conditioning rises only from $1.79$ to $2.33$, a $30\%$ increase for a
realized displacement reaching $17\%$. This exhibits the $h$-uniformity
of~\eqref{eq:main-cond} at the level of the quantity the theorem bounds. We
do not claim to measure $C_{\mathrm{geo}}$ or $\nu_h$: the sharp constant of
Corollary~\ref{cor:hnu} is a supremum in the $\bs H^{-1/2}(\dive_\Gamma)$
norm, whose discrete Gram matrix is not available here.

\begin{figure}[t]
  \centering
  \includegraphics[width=\textwidth]{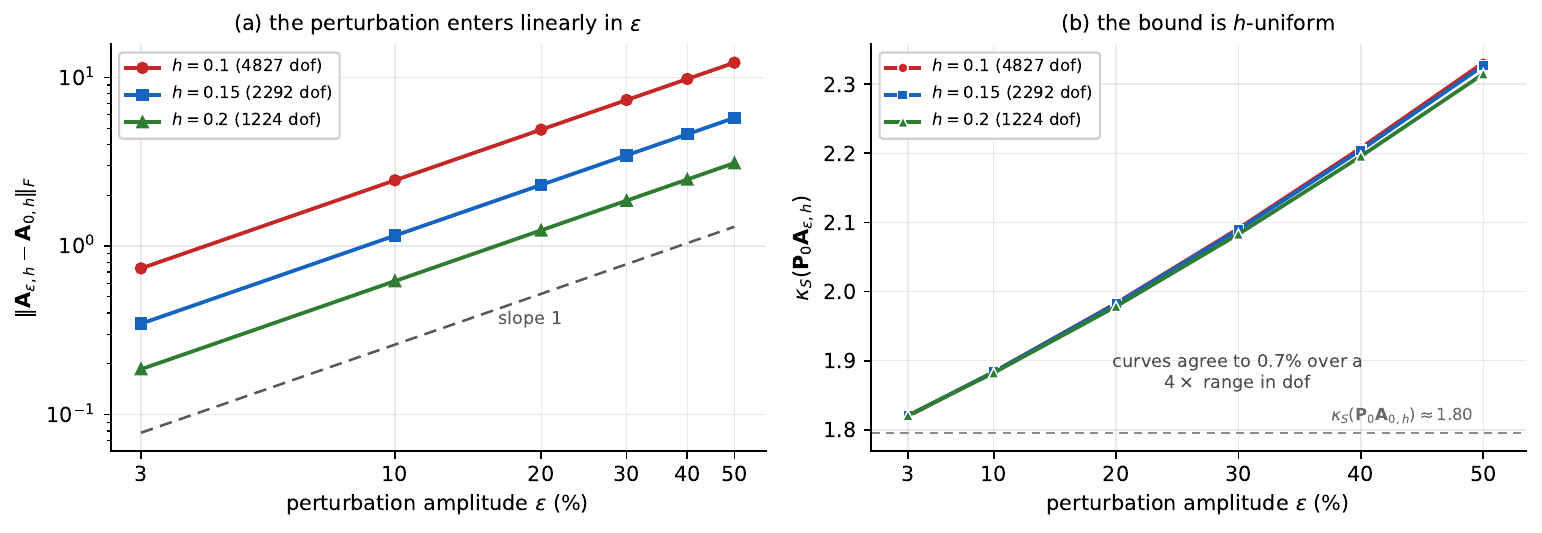}
  \caption{Direct test of the perturbation estimate, nominal sphere,
    $\kappa = 2$, single realization (seed 1001), at three mesh levels.
    (a)~$\|\bs A_{\varepsilon,h}-\bs A_{0,h}\|_F$ against the perturbation
    amplitude, with a slope-one reference: the least-squares slopes are
    $1.001$, $0.999$ and $0.999$ at $h = 0.2, 0.15, 0.1$.
    (b)~$\kappa_S(\bs P_0\bs A_{\varepsilon,h})$,
    the quantity bounded in~\eqref{eq:main-cond}, computed from the
    eigenvalues of the frozen-preconditioned matrix; the curves agree to
    $0.8\%$ over a fourfold range in degrees of freedom.}
  \label{fig:linearity}
\end{figure}

\subsection{Mesh refinement at fixed perturbation}
\label{sec:refinement}
Table~\ref{tab:href}
provides direct numerical evidence of $h$-uniform behavior:
fixing a moderate deformation (sphere, amplitude $30\%$,
realization~1, $\kappa=2$) and sweeping $h$ from $0.2$ down to $0.075$
($1224$ to $8172$ dof), plain GMRES deteriorates steadily with
refinement ($168$ to $346$ iterations), the expected
first-kind-operator behavior, while the frozen preconditioner's
iteration count is \emph{exactly} $13$ at every mesh level, matching
the freshly assembled preconditioner to within one iteration
throughout. Since the computations use the straight-sided
discretization of Remark~\ref{rmk:variationalcrime}, this behavior
also persists under the geometric approximation lying outside the
theorem's literal scope.

\begin{table}[t]
  \centering
  \scriptsize
  \caption{Mesh refinement at fixed perturbation (sphere, amplitude
    $30\%$, realization 1, $\kappa=2$). Frozen iteration counts are
    flat in $h$, consistent with the $h$-uniform bound of
    Theorem~\ref{thm:main}, while plain GMRES deteriorates. Last three
    columns: converged true relative GMRES residual.}
  \label{tab:href}
  \begin{tabular}{@{}rrrrrrrr@{}}
    \toprule
    & & & & & \multicolumn{3}{c}{True residual} \\
    $h$ & dof & Plain it. & Frozen it. & Fresh it.
      & plain & frozen & fresh \\
    \midrule
    0.2   & 1224 & 168 & 13 & 13 & 1.0e-8 & 7.5e-8 & 6.6e-8 \\
    0.15  & 2292 & 210 & 13 & 12 & 8.5e-9 & 1.9e-8 & 7.0e-8 \\
    0.1   & 4827 & 281 & 13 & 12 & 9.9e-9 & 3.6e-8 & 5.9e-8 \\
    0.075 & 8172 & 346 & 13 & 12 & 9.9e-9 & 5.4e-7 & 5.4e-7 \\
    \bottomrule
  \end{tabular}
\end{table}

\subsection{Stochastic Galerkin \texorpdfstring{$p$}{p}-refinement}
\label{sec:sgexp}
To test the stochastic-dimension uniformity of
Corollary~\ref{cor:bochner}, we solve the \emph{coupled} stochastic
Galerkin system for the sphere at $h = 0.2$, $\kappa = 2$,
$\varepsilon = 20\%$, with $d = 3$ random modes and total-degree Legendre
chaos spaces
$S_p$, $p = 1,\dots,4$, with $L^2_\rho$-orthonormal basis. The modes are
the $(l,m) = (2,-2), (2,-1), (2,0)$ coefficients of~\eqref{eq:pertformula},
each scaled by $\varepsilon y_i/l^2$ with $\bs y \in [-1,1]^3$ under the
product uniform measure. The
stochastic Galerkin blocks and right-hand side are computed with
tensorized $5$-point Gauss--Legendre quadrature ($125$ nodes, one
DOF-aligned EFIE assembly per node); the coupled operator is applied
matrix-free from the stored nodal matrices. Since $\bs A_h(\varepsilon,\bs
y)$ is non-polynomial in $\bs y$ the rule is not exact, so the computed
system is a quadrature-assembled approximation; repeating $p=4$ with seven
points per parameter ($343$ nodes) gave identical iteration counts and true
residuals agreeing to six significant digits, so quadrature error is
immaterial here. For this experiment $\bs P_0$ is applied exactly, via a
direct factorization of $\bs N_{XY,0}$, to isolate stochastic-discretization
effects from the inexact inner solves used elsewhere.
Table~\ref{tab:sgexp} reports the counts: $I \otimes \bs P_0$ requires
\emph{exactly} $12$ iterations at every chaos degree while the coupled
system grows from $4\,896$ to $42\,840$ unknowns and plain GMRES
deteriorates from $230$ to $279$, the stochastic-dimension-independent
robustness predicted by Corollary~\ref{cor:bochner}.

\begin{table}[t]
  \centering
  \scriptsize
  \caption{Stochastic Galerkin $p$-refinement (sphere, $h=0.2$,
    $\kappa=2$, $\varepsilon=20\%$, $d=3$ random modes, total-degree
    Legendre chaos). Iterations of the block-diagonal frozen
    preconditioner $I \otimes \bs P_0$ are flat in the stochastic
    dimension. Last two columns: converged true relative residual.}
  \label{tab:sgexp}
  \begin{tabular}{@{}rrrrrrr@{}}
    \toprule
    & & & & & \multicolumn{2}{c}{True residual} \\
    $p$ & $\dim S_p$ & SG dofs & Plain it. & $I \otimes \bs P_0$ it.
      & plain & $I \otimes \bs P_0$ \\
    \midrule
    1 & 4  & 4\,896  & 230 & 12 & 1.0e-8 & 2.1e-9 \\
    2 & 10 & 12\,240 & 269 & 12 & 9.9e-9 & 6.9e-9 \\
    3 & 20 & 24\,480 & 276 & 12 & 9.7e-9 & 8.1e-9 \\
    4 & 35 & 42\,840 & 279 & 12 & 9.8e-9 & 1.0e-8 \\
    \bottomrule
  \end{tabular}
\end{table}

\subsection{Wavenumber: higher frequency and mismatch}
\label{sec:kappamismatch}
Repeating the amplitude sweep at $\kappa=4$ (same $h$, mesh,
realization) shows the same qualitative behavior at a higher, more
oscillatory regime: plain GMRES requires $260$--$317$ iterations across
the sweep; frozen requires $22$--$34$; fresh requires $21$--$33$, a
gap of at most one iteration throughout --- comparable to the range at
$\kappa=2$, although the Calder\'on count roughly doubles, so over these two
wavenumbers the freezing penalty does not scale with the difficulty of the solve
(the constants themselves depend on $\kappa$; Remark~\ref{rmk:regularity}).

Freezing the preconditioner in \emph{wavenumber} rather than in
\emph{geometry}, however, behaves differently, as anticipated in
Remark~\ref{rmk:kappa-mismatch}. Fix the preconditioner at
$\kappa_{\mathrm{precond}}=2$ on the nominal, unperturbed sphere and sweep
$\kappa_{\mathrm{solve}}$ over six octaves below it and two above
(Table~\ref{tab:kappamismatch}). The penalty
$\mathrm{iters}_{\mathrm{mismatched}}/\mathrm{iters}_{\mathrm{fresh}}$ is
not symmetric in the two directions. Upward it grows with no sign of an
asymptote, reaching $6.5\times$ at a ratio of $4$. Downward it grows and then
\emph{saturates}, settling near $4.6\times$ over the last three octaves with
the frozen count at $32$--$33$ against a freshly assembled count pinned at
$7$. The plateau appears for every preconditioner wavenumber tested, at a
level fixed by $\kappa_{\mathrm{precond}}$ alone: freezing at
$\kappa_{\mathrm{precond}} = 1, 2, 4$ and solving at
$\kappa_{\mathrm{solve}} = 0.03125$ costs $4.0\times$, $4.6\times$ and
$11.6\times$.
The plateau is the signature of the low-frequency breakdown of the
{EFIE}. As $\kappa \to 0$ the $1/(\iota\kappa)$ weight on the
scalar-potential term of~\eqref{eq:efie} comes to dominate the
vector-potential one. 
The operator itself does not converge in that limit; it diverges, precisely
because of that weight. What converges is $\iota\kappa\,\calT$, which tends to
the scalar-potential term alone, so that $\calT$ differs from a fixed limiting
operator by a divergent scalar factor. The relative residual reduction of
GMRES is invariant under scalar multiplication of the system, and the frozen
preconditioner is fixed, so the iteration count is governed by that limiting
operator and must settle at a value determined by
$\kappa_{\mathrm{precond}}$, which is what Table~\ref{tab:kappamismatch}
records. No such limit exists as $\kappa$ grows, and none is observed. The
same imbalance predicts the
accuracy loss reported below, the two potential contributions being resolved
at very different relative precisions once the weights are so far apart. The
breakdown is classical and is what the loop-star and loop-tree decompositions
were introduced to
cure~\cite{Vecchi1999}; the configuration closest
to ours, loop-star combined with a Calder\'on multiplicative preconditioner,
is treated in~\cite{YanJinNie2010}. We invoke it as a diagnosis rather than
adopt it as a remedy: a loop-star basis addresses the breakdown itself and not
the mismatch penalty at issue here.

Two controls bear on that reading. The interior Maxwell spectrum of the unit
ball begins at $\kappa \approx 2.7437$, so the downward arm is
resonance-free, whereas $\kappa_{\mathrm{solve}} = 4$ and $8$ lie within
$3.2\%$ and $2.3\%$ of an interior eigenvalue; repeating those two solves at
the detuned $\kappa = 3.52$ and $\kappa = 8.4661$ places them on the same
curve. And plain GMRES stays within $256$--$369$ iterations across the whole
sweep, so the penalty does not track a deteriorating unpreconditioned
problem.
Unlike the geometric case, this behavior is characterized here
only numerically. At the bottom of the sweep the binding limitation is
accuracy rather than iteration count: at $\kappa_{\mathrm{solve}} = 0.03125$
the frozen and freshly preconditioned solutions agree only to
$2.8\times10^{-5}$ in relative $\ell^2$ norm (preconditioner frozen at
$\kappa_{\mathrm{precond}} = 4$), against $10^{-8}$ in the mid-range, and we
draw no conclusion there about the attainable solution. This is the
accuracy consequence of the low-frequency breakdown identified above, a
property of the formulation rather than of preconditioner reuse.

\begin{table}[t]
  \centering
  \scriptsize
  \caption{Bidirectional wavenumber mismatch, nominal sphere, $h=0.1$,
    preconditioner frozen at $\kappa_{\mathrm{precond}}=2$. The penalty
    saturates below the matched wavenumber and does not above it. Rows
    marked $^{\dagger}$ are the detuned resonance controls of the text;
    rows marked $^{\ast}$ lie within $5\%$ of an interior Maxwell
    eigenvalue of the unit ball (the spectrum begins at $\kappa \approx
    2.7437$, so every $\kappa_{\mathrm{solve}} \le 2$ is resonance-free).}
  \label{tab:kappamismatch}
  \begin{tabular}{@{}rrrrrr@{}}
    \toprule
    $\kappa_{\mathrm{solve}}$ &
      $\log_2(\kappa_{\mathrm{solve}}/\kappa_{\mathrm{precond}})$ &
      Plain iters & Mismatched iters & Fresh iters & Penalty \\
    \midrule
    0.03125    & $-6$    & 369 & 32  & 7  & 4.6$\times$ \\
    0.0625     & $-5$    & 360 & 33  & 7  & 4.7$\times$ \\
    0.125      & $-4$    & 337 & 32  & 7  & 4.6$\times$ \\
    0.25       & $-3$    & 317 & 28  & 7  & 4.0$\times$ \\
    0.5        & $-2$    & 291 & 24  & 8  & 3.0$\times$ \\
    1          & $-1$    & 275 & 18  & 8  & 2.3$\times$ \\
    2          & \phantom{$-$}0 & 272 & 9   & 9  & 1.0$\times$ \\
    3.52$^{\dagger}$ & \phantom{$-$}0.82 & 297 & 29  & 13 & 2.2$\times$ \\
    4$^{\ast}$ & \phantom{$-$}1 & 256 & 46  & 16 & 2.9$\times$ \\
    8$^{\ast}$ & \phantom{$-$}2 & 353 & 162 & 25 & 6.5$\times$ \\
    8.4661$^{\dagger\ast}$ & \phantom{$-$}2.08 & 305 & 173 & 26 & 6.7$\times$ \\
    \bottomrule
  \end{tabular}
\end{table}

\subsection{Ellipsoid: a non-spherical geometry}
\label{sec:ellipsoid}
Since the sphere is highly symmetric, we repeat the amplitude sweep on a
triaxial ellipsoid, semi-axes $(a,b,c)=(1.5,1.0,0.7)$ (no two equal, so no
residual rotational symmetry beyond the discrete sign-flip group), under the
deformation family of Section~\ref{sec:numerics}. Iteration counts are stable
across the whole sweep (Table~\ref{tab:ellipsoid}): plain GMRES requires
$322$--$325$ iterations while both frozen and fresh Calder\'on
preconditioning require $16$--$17$, a gap of at most one iteration at every
amplitude and zero at $\varepsilon=20\%$. The robustness seen on the sphere
is therefore not a consequence of spherical symmetry. In the convention
of Section~\ref{sec:sphereamp}, the realized displacement at
$\varepsilon=30\%$ reaches $10.28\%$ of the local ellipsoid radius, about
$1.5h$ and $\lambda/20$.
\begin{table}[t]
  \centering
  \scriptsize
  \caption{Triaxial ellipsoid ($a{=}1.5,b{=}1.0,c{=}0.7$), amplitude
    sweep, $\kappa=2$, $h=0.1$ (dof $=4827$, matching the sphere case),
    single realization (seed 1001). All meshes valid (minimum triangle
    area $>0$) throughout. Last three columns: converged true relative
    GMRES residual $\|\bs b-\bs Z\bs u\|/\|\bs b\|$.}
  \label{tab:ellipsoid}
  \begin{tabular}{@{}rrrrrrrr@{}}
    \toprule
    & & & & Frozen & \multicolumn{3}{c}{True residual} \\
    Amplitude & Plain it. & Frozen it. & Fresh it. & $-$ fresh
      & plain & frozen & fresh \\
    \midrule
    3\%  & 325 & 16 & 15 & 1 & 8.7e-9 & 6.0e-8 & 6.1e-8 \\
    10\% & 324 & 16 & 15 & 1 & 8.9e-9 & 6.1e-8 & 6.0e-8 \\
    20\% & 323 & 16 & 16 & 0 & 8.8e-9 & 6.2e-8 & 5.7e-8 \\
    30\% & 322 & 17 & 16 & 1 & 9.8e-9 & 6.3e-8 & 5.4e-8 \\
    \bottomrule
  \end{tabular}
\end{table}

\subsection{Fichera corner geometry: beyond the verified regime}
\label{sec:fichera}
The sphere and ellipsoid are both smooth and convex, and are covered
(up to the geometric variational crime of
Remark~\ref{rmk:variationalcrime}) by the $C^{1,1}$-domain verification
of Assumption~\ref{ass:geocont} provided by
Lemma~\ref{lem:shapepert}. We now step \emph{outside} that verified
regime and repeat the experiment on a \emph{Fichera corner geometry},
the solid $[-1,1]^3 \setminus [0,1]^3$, whose single reentrant corner
makes it both non-convex and non-smooth, so that neither the
$C^{1,1}$ boundary hypothesis of Lemma~\ref{lem:shapepert} nor the
smooth-surface reading of the Calder\'on identity
(Section~\ref{sec:calderon}) applies. The geometry is meshed directly
with Gmsh's OpenCASCADE boolean-subtraction kernel; the resulting
surface mesh is watertight, as verified from its Euler characteristic
and edge incidence.

The perturbation also differs: rather than a global spherical-harmonic
radial field, we apply a spatially \emph{localized} bump/dent to two of the
cube's nine flat faces ($x=-1$ and $y=-1$), displaced along each face's
outward normal by $\varepsilon\cos^2(\pi y/2)\cos^2(\pi z/2)$ and
$\varepsilon\cos^2(\pi x/2)\cos^2(\pi z/2)$. The taper vanishes exactly at
each perturbed face's own boundary, so vertices shared with neighbouring
faces --- including the reentrant corner itself --- never move, mesh
connectivity is identical between nominal and perturbed meshes, and the
DOF-alignment strategy of Section~\ref{sec:implementation} applies
unchanged. The two faces were chosen away from the reentrant corner so that
the perturbation's effect separates cleanly from the geometry's intrinsic
singularity.

Fixing $\kappa=2$, $h=0.3$ (dof$\,{=}\,1212$, a coarser discretization
than the sphere or ellipsoid studies, adequate for this proof-of-concept
third geometry), a sweep of the bump amplitude from $2\%$ to $15\%$ of
the cube's half-width returns a constant $197$ plain iterations and
\emph{exactly} $26$ for both the frozen and the freshly assembled
preconditioner at every amplitude, a zero gap throughout, with true
residuals near $6\times10^{-8}$ in all three solves. The largest realized
displacement is $0.143$, that is $0.48h$ and about $\lambda/22$ here: this
is the only one of the three sweeps whose displacement stays below a single
mesh width, which should be kept in mind when reading that zero gap.
The Fichera corner is Lipschitz but not $C^{1,1}$, its edges and
reentrant corner leaving the normal field discontinuous. It therefore lies
outside the verified range of Assumption~\ref{ass:geocont}, and
Theorem~\ref{thm:main} does not apply.
Refining to $h = 0.15$ ($4170$ unknowns) changes nothing: $289$ plain
iterations and $26$ frozen $=$ $26$ fresh at every amplitude, with the
minimum triangle area unchanged at $5.61\times10^{-3}$. A perturbation
confined to a flat face, away from the reentrant corner, does not interact
with the geometric singularity at all, and the zero gap is evidence about
the perturbation rather than about the geometry.

We therefore repeated the experiment with the perturbation moved onto the
singularity. The three reentrant notch faces $\{x=0\}$, $\{y=0\}$ and
$\{z=0\}$, which meet pairwise along the reentrant edges and all three at
the reentrant corner, are displaced along their outward normals by
$\varepsilon\,g(u)\,g(v)$ in the two in-face coordinates, with
$g(t) = \cos^2(\pi t/2)$. Since $g(0)=1$ and $g(1)=g'(1)=0$, the
displacement vanishes with vanishing slope where each notch meets an outer
face, leaving the six outer faces exactly flat, and is \emph{maximal at the
reentrant corner}, which moves by $\sqrt3\,\varepsilon$ along the body
diagonal. A vertex lying on more than one notch face receives the vector
sum of the corresponding contributions, each along that face's own outward
normal, so its displacement is single-valued: a vertex on a reentrant edge
moves diagonally, in the plane spanned by the two normals, and the dihedral
angle along that edge therefore changes with $\varepsilon$. Connectivity is
preserved exactly as before, and the mesh,
wavenumber and discretization are those of the control just described, so
the two sweeps differ only in where the perturbation acts.

Table~\ref{tab:ficheracorner} reports the outcome, and it differs in kind.
The frozen--fresh gap, zero at every amplitude on the flat faces, grows to
$1$, $4$, $13$ and $23$ iterations, and the ratio of frozen to fresh counts
grows monotonically from $1.04$ to $1.51$. The contrast with the smooth
geometries is sharpest at matched realized displacement: the sphere's
largest case displaces $1.71h$, or $\lambda/18$, for a gap of at most two
iterations, whereas here the gap is $13$ at $\lambda/18$ and $23$ at
$1.73h$.

Mesh degradation and operator perturbation both act here, and the two must
be separated. The cleanest evidence is the $\varepsilon=5\%$ row, where the
mesh is only mildly affected: the shape-regularity ratio rises from $2.29$
to $3.05$ and the smallest angle from $42.7^\circ$ to $31.8^\circ$, values
unremarkable for a surface mesh, plain GMRES moves by five iterations
($296$ against $291$), and the freshly assembled preconditioner is
unchanged at $25$. The problem, in other words, is not yet appreciably
harder. Yet the frozen gap is already $4$, against $0$ at the same
amplitude in the flat-face control. That gap is not attributable to element
quality.

At $\varepsilon=10\%$ and $15\%$ the two effects are genuinely confounded:
the smallest angle falls to $10.0^\circ$ and then $2.8^\circ$, with
$\sigma_{\max}$ reaching $33$, and plain GMRES
and the fresh preconditioner both deteriorate, to $446$ and $45$
iterations. The minimum triangle area is the wrong diagnostic here, and
misleadingly so: it \emph{rises} from $1.5\times10^{-4}$ to
$2.0\times10^{-4}$ between those two amplitudes, while every shape measure
worsens. What still separates them is the \emph{ratio} of frozen to
fresh counts, which rises monotonically from $1.04$ to $1.51$: a difficulty
shared by both strategies moves the two counts together and leaves the
ratio near unity, whereas freezing costs progressively more. The observed
deterioration is consistent with the failure of a small operator-norm
perturbation bound at a moving reentrant corner --- the behaviour controlled
by Assumption~\ref{ass:geocont} and for which Lemma~\ref{lem:shapepert}
provides the required estimate only on $C^{1,1}$ boundaries.

Read positively, this locates what Assumption~\ref{ass:geocont} controls.
Reuse is robust on the smooth geometries the theory covers; it is equally
robust on a geometry the theory does \emph{not} cover, at every amplitude
tested, so long as the perturbation leaves the singularity untouched; and,
in this sweep, it degrades monotonically once the perturbation acts at the
singularity itself. Taken together, these experiments indicate that the
obstruction is not non-smoothness of the scatterer alone, but perturbation of
a \emph{reentrant} singularity --- a sharper empirical distinction than the
$C^{1,1}$ hypothesis itself makes, and one that suggests what a future
extension would need to track. Establishing
Assumption~\ref{ass:geocont} for Lipschitz or polyhedral Maxwell boundaries
remains open (Section~\ref{sec:conclusions}), and these counts suggest that
any such result will have to track the corner explicitly.

\begin{table}[t]
  \centering
  \scriptsize
  \caption{Fichera corner geometry, perturbation applied \emph{at} the
    reentrant corner (the three notch faces $x=0$, $y=0$, $z=0$; amplitude
    $\varepsilon$ expressed as a percentage of the cube's half-width, the
    corner itself moving by $\sqrt3\,\varepsilon$), $\kappa=2$,
    $h=0.15$ (dof $=4170$). The matched flat-face control at the same $h$
    and the same dof gives $289$ plain and $26$ frozen $=$ $26$ fresh
    iterations at every amplitude, a zero gap throughout. All meshes remain
    valid. $d_{\mathrm{c}}$ is the measured corner displacement;
    $\theta_{\min}$ is the smallest interior angle over all triangles and
    $\sigma_{\max}$ the largest ratio of longest edge to twice the in-radius,
    which is $\sqrt3$ for an equilateral triangle and diverges as an element
    degenerates. The nominal mesh has $\theta_{\min}=42.7^\circ$ and
    $\sigma_{\max}=2.29$.}
  \label{tab:ficheracorner}
  \begin{tabular}{@{}rrrrrrrrr@{}}
    \toprule
    & & & & & & & Frozen & \\
    $\varepsilon$ & $d_{\mathrm{c}}$ & in $h$ & $\theta_{\min}$ &
      $\sigma_{\max}$ & Plain it.
      & Frozen it. & Fresh it. & $-$ fresh \\
    \midrule
    2\%  & 0.035 & 0.23 & $40.6^\circ$ & 2.35 & 291 & 26 & 25 & 1 \\
    5\%  & 0.087 & 0.58 & $31.8^\circ$ & 3.05 & 296 & 29 & 25 & 4 \\
    10\% & 0.173 & 1.15 & $10.0^\circ$ & 8.63 & 327 & 42 & 29 & 13 \\
    15\% & 0.260 & 1.73 & $\phantom{0}2.8^\circ$ & 33.09 & 446 & 68 & 45 & 23 \\
    \bottomrule
  \end{tabular}
\end{table}

\subsection{NASA almond: a Monte Carlo shape-uncertainty campaign}
\label{sec:almond}

The geometries above are small and smooth, or small and singular. The NASA
almond of Woo et al.~\cite{WooWangSchuhSanders1993} is neither. It is a
standard radar-cross-section benchmark: convex, strongly non-uniform, with a
rounded nose at one end and a sharp tip at the other, and a maximum
half-thickness of $16.26$\,mm against a total length $d = 9.936$\,in
$=0.252374$\,m. We take $d = 3\lambda$, so that $\lambda = 84.125$\,mm,
$f = 3.564$\,GHz and $\kappa d = 18.85$, by a wide margin the electrically
largest target in the paper. Its surface is meshed at
$h = \lambda/12$ ($893$ vertices, $1782$ faces, $2673$ RWG unknowns) and, for
the refinement check below, at $h = \lambda/20$ ($7269$ unknowns).\footnote{%
The published profile's two arcs do not quite meet at the mid-station,
leaving a radial jump of $5\times10^{-6}d \approx 1.3\,\mu$m --- an artifact
of rounding in the published constants, four orders of magnitude below the
mesh width.}

This experiment draws its shapes at random, so admissibility in the sense of
Section~\ref{sec:shapeperts} must be established rather than assumed. We
therefore perturb by that class itself, $\bs\varphi_{\bs y}(\bs x) = \bs x +
\varepsilon\bs V(\bs x)$ with $\bs V$ a vector-valued random field on $\R^3$
--- three independent squared-exponential components of correlation length
$\ell$, realized by random Fourier features --- rather than by displacing
vertices along surface normals. Two things follow. First, $\bs V$ is analytic
with a closed-form Jacobian, so $\bs\varphi_{\bs y}$ is injective on any convex
set as
soon as $\varepsilon\sup|\mathrm D\bs V|_2 < 1$; we impose
$\varepsilon\sup|\mathrm D\bs V|_2 \le \tfrac12$, evaluating the supremum on
a refined grid over the bounding box of $\Gnom$. Validity is therefore
screened \emph{before} a shape is drawn: no realization is ever discarded
and the sample carries no rejection bias, whereas rejecting self-intersecting
draws after the fact would condition the sample on an event correlated with
the amplitude. Second, the sharp tip becomes a variable rather than a given:
case~U applies $\bs V$ unchanged, so the tip moves with its neighbourhood,
while case~T multiplies it by a $C^1$ cutoff vanishing within $d/10$ of the
tip --- this geometry's version of the question
Section~\ref{sec:fichera} asks at the reentrant corner.
Figure~\ref{fig:mesh-comparison} shows one draw at the largest amplitude
beside the other three geometries.

At $\ell = d/4$ the bound admits $\varepsilon \le 13.30$\,mm over the $110$
fields drawn, and the five amplitude levels are $\varepsilon = 2.66$, $5.32$,
$7.98$, $10.64$ and $13.30$\,mm, that is $\lambda/31.6$ down to
$\lambda/6.3$, or $0.38h$ to $1.90h$, or $16.4\%$ to $81.8\%$ of the seam
half-thickness, exceeding in mesh widths both the sphere at $50\%$ ($1.71h$)
and the Fichera corner at $15\%$ ($1.73h$). The correlation length is a lever
on that ceiling: a longer $\ell$ admits a larger $\varepsilon$ because
$\sup|\mathrm D\bs V|$ falls, so amplitude is bought with a smoother field,
which is also the easier field for the preconditioner. We therefore report
$\ell = d/10$ alongside, at the two levels its own screening admits. All
$243$ geometries below pass a triangle--triangle intersection test, retained
as a check on the implementation rather than as a filter.

Table~\ref{tab:almondtiming} shows the same cost structure as on the sphere,
giving $17.4\times$ measured, $22.9\times$ artifact-corrected and
$27.4\times$ assembly-only for the three figures of
Section~\ref{sec:numanalysis}. These are smaller than the sphere's because
$t_T/t_Z = 27.9$ here against $37$ there: the mechanism is the cost model's,
the factor is geometry- and mesh-dependent.

\begin{table}[t]
  \centering
  \caption{Per-matrix assembly time on the NASA almond,
    $\kappa = 74.69\,\mathrm{m}^{-1}$, $h = \lambda/12$ (dof $=2673$).
    Compare Table~\ref{tab:componenttiming}.}
  \label{tab:almondtiming}
  \begin{tabular}{lrr}
    \toprule
    Matrix & Time (s) & Share \\
    \midrule
    $\bs Z_{XX}$ (primal EFIE) & 7.32 & 3.4\% \\
    $\bs T_{YY}$ (dual, barycentric) & 204.03 & 95.6\% \\
    $\bs N_{XY}$ (duality pairing) & 1.51 & 0.7\% \\
    $\bs b_X$ & 0.46 & 0.2\% \\
    \midrule
    Total & 213.32 & \\
    Frozen, per realization ($\bs Z_{XX} + \bs b_X$) & 7.78 & \\
    \bottomrule
  \end{tabular}
\end{table}

\begin{table}[t]
  \centering
  \caption{NASA almond amplitude sweep, $h=\lambda/12$, $\ell = d/4$, ten
    random fields per cell, case~U. Iteration counts are medians over the
    cell, with the range in brackets; the fresh column is over the samples
    per cell solved both ways --- three at levels~1--3, five at levels~4
    and~5, where the cell's two hardest draws were solved both ways as well.
    Case~T differs nowhere by more than one iteration in the frozen count and
    is discussed below rather than tabulated.}
  \label{tab:almondsweep}
  \begin{tabular}{lrrrrr}
    \toprule
    & $\varepsilon$ & & \multicolumn{3}{c}{GMRES iterations} \\
    \cmidrule(lr){4-6}
    Level & (mm) & $\varepsilon\sup|\mathrm D\bs V|$
      & Plain & Frozen & Fresh \\
    \midrule
    1 & 2.7 & 0.05--0.09 & 418 [409--425] & 48 [47--49] & 47 [47--48] \\
    2 & 5.3 & 0.10--0.19 & 420 [402--436] & 48 [47--50] & 47 [47--49] \\
    3 & 8.0 & 0.15--0.28 & 422 [396--453] & 49 [48--53] & 48 [47--50] \\
    4 & 10.6 & 0.20--0.37 & 424 [395--510] & 50 [48--59] & 51 [47--64] \\
    5 & 13.3 & 0.24--0.47 & 429 [401--473] & 51 [48--55] & 51 [47--60] \\
    \bottomrule
  \end{tabular}
\end{table}

Table~\ref{tab:almondsweep} and Figure~\ref{fig:almondsweep} cover a hundred
samples, fifty in each of cases~U and~T. Across them the frozen-to-fresh
ratio stays within $0.92$--$1.04$ with median exactly $1.00$ over the $38$
solved both ways ($19$ per case), with no systematic amplitude trend; frozen is the cheaper
on $20$ of the $38$. Median counts change little while the maximum rises from
$49$ to $59$, so increasing amplitude primarily lengthens the upper tail.

That tail is the geometry, not the preconditioner. The six samples with the
largest frozen counts are exactly the six with the largest plain counts,
correlated at $0.86$, and the frozen-to-plain ratio is $0.117$ with standard
deviation $0.003$ --- an $8.6$-fold reduction delivered sample by sample
rather than on average. Solving the two hardest draws of each cell at
levels~4 and~5 with a freshly assembled preconditioner confirms it: on all
eight, four per case, the frozen operator was the cheaper, the sweep's worst sample needing
$59$ frozen iterations against $64$ fresh. No sample was flagged as a
resonance candidate, and none failed to converge.

Shortening the correlation length from $d/4$ to $d/10$ leaves this unchanged:
over the two levels its screening admits, the rougher field costs at most
one additional iteration at equal amplitude, consistent with a preconditioner
that responds to the size of the deformation rather than to its spectral
content.

\begin{figure}[t]
  \centering
  \includegraphics[width=\textwidth]{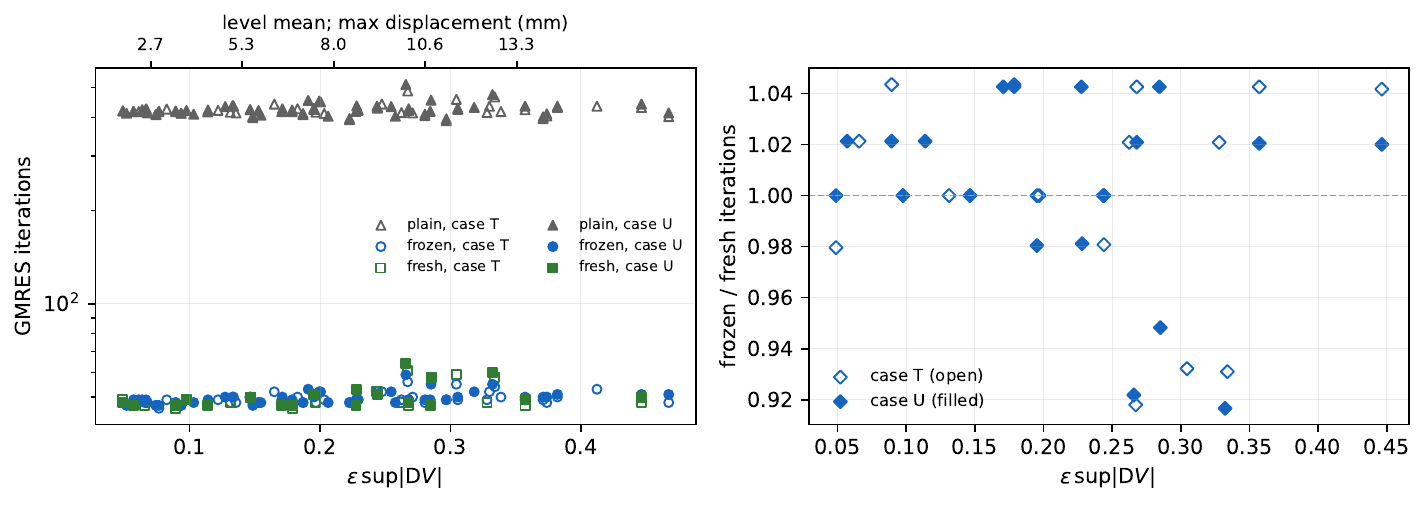}
  \caption{NASA almond amplitude sweep, $h = \lambda/12$, $\ell = d/4$, the
    hundred samples of Table~\ref{tab:almondsweep}. The abscissa is the
    injectivity screening quantity $\varepsilon\sup|\mathrm D\bs V|$, the one
    the perturbation theory is written in, which varies from draw to draw at
    a fixed level; the realized maximum displacement runs along the top axis.
    (a)~iteration counts, plain against frozen against fresh; (b)~the
    frozen-to-fresh ratio, cases~U and~T separately, flat in that quantity
    across its whole admissible range.}
  \label{fig:almondsweep}
\end{figure}

At the top amplitude, $\varepsilon = 13.3$\,mm, one hundred independent
fields were drawn and solved frozen (Figure~\ref{fig:almondcampaign}): median
$51$, range $46$--$58$, standard deviation $2.3$, all converged to the
$10^{-8}$ tolerance, and on the twenty also solved fresh the frozen-to-fresh
ratio is $0.93$--$1.11$ with median $1.02$; no realization was discarded.
That the ratio falls below one on part of the sample is the substantive
point: at this amplitude freezing is a perturbation of either sign about
unity, not a controlled degradation.

\begin{figure}[t]
  \centering
  \includegraphics[width=\textwidth]{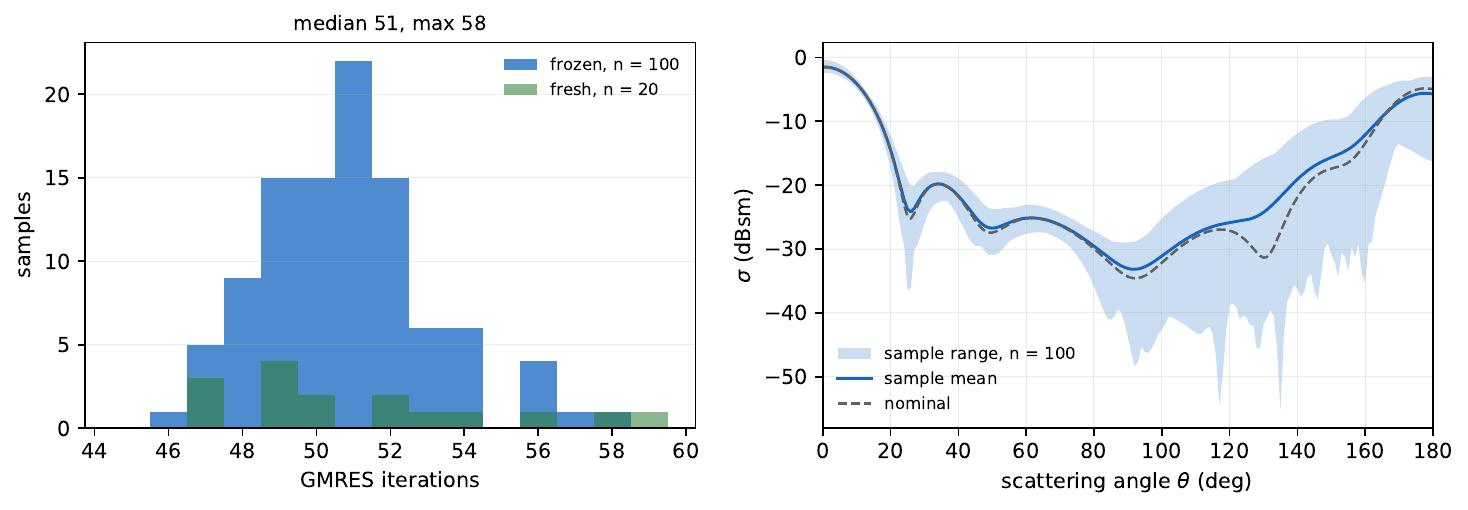}
  \caption{NASA almond Monte Carlo campaign, one hundred independent fields
    at $\varepsilon = 13.3$\,mm $= \lambda/6.3$. (a)~the distribution of
    frozen GMRES iteration counts, with the nominal count and the freshly
    assembled subset marked. (b)~the bistatic radar cross section: the
    nominal cut against the sample envelope over the campaign.}
  \label{fig:almondcampaign}
\end{figure}

Cases~U and~T are indistinguishable: their median frozen counts differ by at
most one iteration at every amplitude, including the top level, where the tip
moves by $4.90$\,mm under~U and is fixed under~T. With
Section~\ref{sec:fichera}, this favours reentrancy rather than
non-smoothness as such as what makes that geometry the harder of the two.

At maximum amplitude, refinement from $2673$ to $7269$ unknowns raises the
unpreconditioned median from $429$ to $617$ while leaving the frozen count
essentially unchanged ($51$ to $52$; frozen-to-fresh ratio $0.98$--$1.06$).
The mesh-independent behaviour targeted by Calder\'on preconditioning
therefore survives freezing, and the advantage in iteration count grows with
refinement, from $8.4\times$ to $11.9\times$.

\begin{figure}[t]
  \centering
  \includegraphics[width=\textwidth]{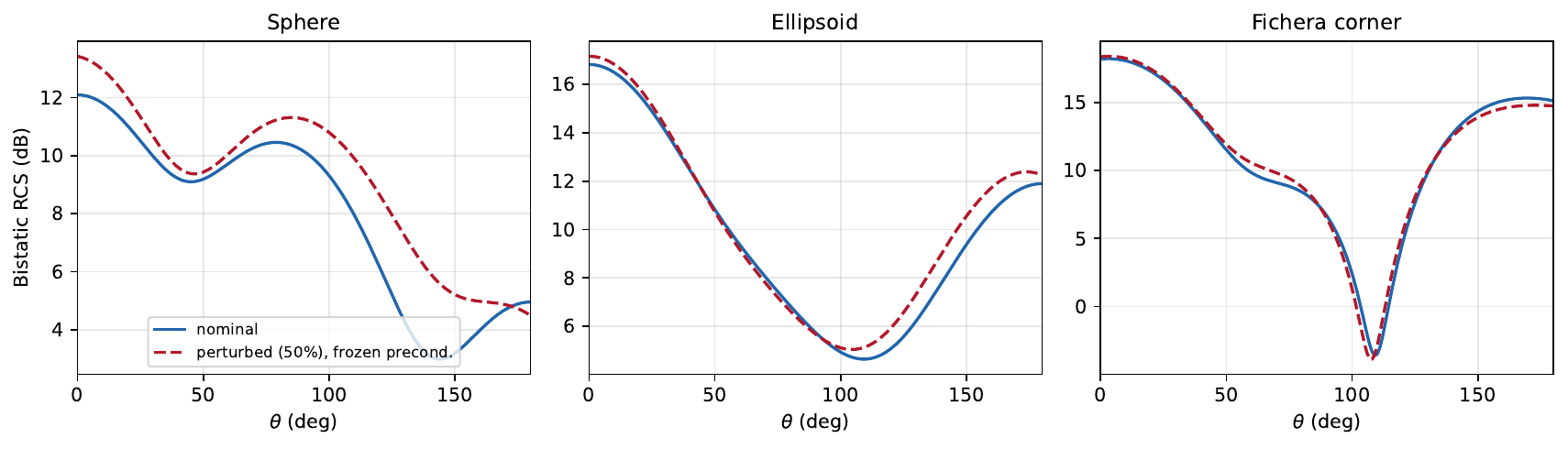}
  \caption{Bistatic radar cross section, nominal against perturbed,
    computed from each converged current with the frozen preconditioner:
    sphere at $50\%$, ellipsoid at $30\%$, Fichera corner at $15\%$. The
    perturbation changes the far field by several dB on the first two,
    while the Fichera curves barely separate, as a perturbation confined to
    two flat faces should not.}
  \label{fig:rcs}
\end{figure}

Two questions should be separated. Reuse does not change the converged
solution: solving one perturbed geometry both ways gives currents agreeing to
$\|\bs u_{\mathrm{frozen}}-\bs u_{\mathrm{fresh}}\|/\|\bs u_{\mathrm{fresh}}\|
= 1.9\times10^{-8}$, with the residual columns of
Tables~\ref{tab:sphereamp}, \ref{tab:ellipsoid} and~\ref{tab:ficheracorner}
the direct algebraic check. The perturbation is nonetheless large enough to
matter physically, separating nominal from perturbed far fields by several dB
(Figure~\ref{fig:rcs}); monostatically, over $73$ azimuths at both
polarizations, it moves the radar cross section by a median of $2.2$\,dB and
by as much as $14$\,dB, while frozen and fresh agree to within one iteration
at every angle, $54$ against $53$ at the median.

This is also the one place where the frozen strategy's advantage is at its
weakest. For this single-geometry, $73$-angle sweep the frozen solves cost
$311$\,s against $123$\,s for fresh solves, while fresh assembly adds
$204$\,s: $311$\,s against $327$\,s overall, effectively break-even. The
crossover is near $70$ right-hand sides in the current implementation and
near $200$ with the storage artifact of Section~\ref{sec:campaign} removed.
Freezing therefore targets many geometries with few excitations each rather
than one geometry with many. What such a comparison can
show is bounded by what Theorem~\ref{thm:main} is about: the bound is on the
solver, and an observable quadratic in the solution inherits it only through
the solution error, so these figures exhibit that inheritance rather than
prove it.

\section{Conclusions and outlook}
\label{sec:conclusions}
We have shown that a Calder\'on preconditioner for the EFIE, assembled once
at a nominal geometry, can be reused across an admissible family of
shape-perturbed problems. The spectral-conditioning bound is uniform in the
spatial discretization (Theorem~\ref{thm:main}) and, after tensorization, in
the stochastic approximation dimension (Corollary~\ref{cor:bochner}); the
GMRES rate additionally requires the coercivity-type hypothesis
of~\cite[Assumption~4]{EscapilInchauspeJerezHanckes2021}, which we do not
establish here. The transport itself is exact: fixed connectivity makes the
degree-of-freedom correspondence a permutation, so the duality pairing is
carried over without approximation (Lemma~\ref{lem:nxy-exact}).
Numerically, on the straight-sided discretizations of
Remark~\ref{rmk:variationalcrime}, frozen preconditioning closely tracks
freshly assembled Calder\'on preconditioning under shape, mesh and stochastic
refinement on the smooth geometries, and the NASA almond supplies a
realistic, electrically large validation of the same behaviour
(Section~\ref{sec:almond}). In these experiments the frozen
preconditioner deteriorates when a reentrant singularity is moved, but not
under a large deformation or a sharp feature alone
(Table~\ref{tab:ficheracorner}). Wavenumber reuse behaves differently and is
not covered by the shape-perturbation result
(Section~\ref{sec:kappamismatch}).

Practically, freezing amortizes the dual barycentric assembly $\bs T_{YY}$,
which dominates assembly cost, at a measured $29$--$31\times$ per new
realization and roughly twofold against published compressed-assembly
timings~\cite{EscapilInchauspeJerezHanckes2019,FierroJerezHanckes2020}.
These factors are implemen-tation-dependent; the transferable result is the
amortization mechanism quantified in
Section~\ref{sec:numanalysis}.

Three limitations delimit the result. First, the DOF alignment of Section~\ref{sec:implementation}
presupposes a mesh-morphing perturbation model, so a realization requiring a
topology change, a quality-driven remesh, or a deformation too large for the
nominal connectivity admits no permutation and falls outside the strategy
altogether; this bears on the sampling pipelines
of~\cite{AylwinJerezHanckesSchwabZech2020,
AylwinJerezHanckesSchwabZech2023} whenever mesh quality forces a remesh.
Second, the admissible perturbation window narrows with frequency, since
$C_{\mathrm{geo}}$ and the stability constants entering $K_\star$ both depend
on $\kappa$; the two wavenumbers reported here do not quantify that
dependence. Third, the gain is in assembly and therefore favours many
geometries with few excitations each; the multi-right-hand-side experiment of
Section~\ref{sec:almond} reaches break-even for a single-geometry sweep.

Three directions remain open. First, the geometric hypothesis:
$C_{\mathrm{geo}}$ is not tracked explicitly in $\kappa$ and the nominal
curvature, the surface-lifting analysis of Remark~\ref{rmk:variationalcrime}
is needed to bring the assembled matrices under the exact-surface theorem,
and Assumption~\ref{ass:geocont} remains open for merely Lipschitz or
polyhedral boundaries, where~\cite{DoelzHenriquez2024} suggests a route that
would bring the Fichera-type geometries inside the result's scope. Second,
the reuse principle of Section~\ref{sec:oppg} is not tied to the EFIE:
forthcoming work will address PMCHWT-type
systems~\cite{KleanthousEtAl2022}, local multiple trace
formulations~\cite{HiptmairJerezHanckes2012}, where a perturbation confined
to a subset of interfaces requires refreshing only the affected Calder\'on
blocks, and OSRC-type operator
preconditioners~\cite{FierroPiccardoBetcke2023}; for wavenumber sweeps,
pre-positioning several preconditioners in the spirit
of~\cite{VanHartenScarabosio2026} is a natural complement. Third, assessing
end-to-end savings in Monte Carlo and multilevel
pipelines~\cite{AylwinJerezHanckesSchwabZech2020,
AylwinJerezHanckesSchwabZech2023} and in the tensorized moment systems of
first-order sparse BEM~\cite{EscapilInchauspeJerezHanckes2024} remains future
work.

\section*{Code and data availability}
The Julia implementation is openly available under the MIT licence at
\url{https://github.com/cfjerez-sci/ReusingCalderon}. 


\bibliography{references}

\begin{thebibliography}{10}

\bibitem{AdrianAndriulliEibert2019}
{\sc S.~B. Adrian, F.~P. Andriulli, and T.~F. Eibert}, {\em On a
  refinement-free {C}alder\'on multiplicative preconditioner for the electric
  field integral equation}, Journal of Computational Physics, 376 (2019),
  pp.~1232--1252,
  \href{http://dx.doi.org/10.1016/j.jcp.2018.10.009}{doi:\nolinkurl{10.1016/j.jcp.2018.10.009}}.

\bibitem{AndriulliEtAl2008}
{\sc F.~P. Andriulli, K.~Cools, H.~Ba{\u g}c{\i}, F.~Olyslager, A.~Buffa, S.~H.
  Christiansen, and E.~Michielssen}, {\em A multiplicative {C}alder\'on
  preconditioner for the electric field integral equation}, IEEE Transactions
  on Antennas and Propagation, 56 (2008), pp.~2398--2412,
  \href{http://dx.doi.org/10.1109/TAP.2008.926788}{doi:\nolinkurl{10.1109/TAP.2008.926788}}.

\bibitem{AylwinJerezHanckesSchwabZech2020}
{\sc R.~Aylwin, C.~Jerez-Hanckes, C.~Schwab, and J.~Zech}, {\em Domain
  uncertainty quantification in computational electromagnetics}, SIAM/ASA
  Journal on Uncertainty Quantification, 8 (2020), pp.~301--341,
  \href{http://dx.doi.org/10.1137/19M1239374}{doi:\nolinkurl{10.1137/19M1239374}}.

\bibitem{AylwinJerezHanckesSchwabZech2023}
{\sc R.~Aylwin, C.~Jerez-Hanckes, C.~Schwab, and J.~Zech}, {\em Multilevel
  domain uncertainty quantification in computational electromagnetics},
  Mathematical Models and Methods in Applied Sciences, 33 (2023), pp.~877--921,
  \href{http://dx.doi.org/10.1142/S0218202523500264}{doi:\nolinkurl{10.1142/S0218202523500264}}.

\bibitem{BenziBertaccini2003}
{\sc M.~Benzi and D.~Bertaccini}, {\em Approximate inverse preconditioning for
  shifted linear systems}, BIT Numerical Mathematics, 43 (2003), pp.~231--244.

\bibitem{BuffaChristiansen2007}
{\sc A.~Buffa and S.~H. Christiansen}, {\em A dual finite element complex on
  the barycentric refinement}, Mathematics of Computation, 76 (2007),
  pp.~1743--1769,
  \href{http://dx.doi.org/10.1090/S0025-5718-07-01965-5}{doi:\nolinkurl{10.1090/S0025-5718-07-01965-5}}.

\bibitem{BuffaCostabelSheen2002}
{\sc A.~Buffa, M.~Costabel, and D.~Sheen}, {\em On traces for
  {$\mathbf{H}(\mathrm{curl},\Omega)$} in {L}ipschitz domains}, Journal of
  Mathematical Analysis and Applications, 276 (2002), pp.~845--867,
  \href{http://dx.doi.org/10.1016/S0022-247X(02)00455-9}{doi:\nolinkurl{10.1016/S0022-247X(02)00455-9}}.

\bibitem{BuffaHiptmair2003}
{\sc A.~Buffa and R.~Hiptmair}, {\em Galerkin boundary element methods for
  electromagnetic scattering}, in Topics in Computational Wave Propagation,
  vol.~31 of Lecture Notes in Computational Science and Engineering, Springer,
  Berlin, 2003, pp.~83--124.

\bibitem{ChristiansenNedelec2002}
{\sc S.~H. Christiansen and J.-C. N{\'e}d{\'e}lec}, {\em A preconditioner for
  the electric field integral equation based on {C}alder\'on formulas}, SIAM
  Journal on Numerical Analysis, 40 (2002), pp.~1100--1135,
  \href{http://dx.doi.org/10.1137/S0036142901388731}{doi:\nolinkurl{10.1137/S0036142901388731}}.

\bibitem{CostabelLeLouer2012a}
{\sc M.~Costabel and F.~Le~Lou{\"e}r}, {\em Shape derivatives of boundary
  integral operators in electromagnetic scattering. {P}art {I}: {S}hape
  differentiability of pseudo-homogeneous boundary integral operators},
  Integral Equations and Operator Theory, 72 (2012), pp.~509--535,
  \href{http://dx.doi.org/10.1007/s00020-012-1954-z}{doi:\nolinkurl{10.1007/s00020-012-1954-z}}.

\bibitem{CostabelLeLouer2012b}
{\sc M.~Costabel and F.~Le~Lou{\"e}r}, {\em Shape derivatives of boundary
  integral operators in electromagnetic scattering. {P}art {II}: {A}pplication
  to scattering by a homogeneous dielectric obstacle}, Integral Equations and
  Operator Theory, 73 (2012), pp.~17--48,
  \href{http://dx.doi.org/10.1007/s00020-012-1955-y}{doi:\nolinkurl{10.1007/s00020-012-1955-y}}.

\bibitem{DoelzHenriquez2024}
{\sc J.~D{\"o}lz and F.~Henr{\'i}quez}, {\em Parametric shape holomorphy of
  boundary integral operators with applications}, SIAM Journal on Mathematical
  Analysis, 56 (2024), pp.~6731--6767,
  \href{http://dx.doi.org/10.1137/23M1576451}{doi:\nolinkurl{10.1137/23M1576451}}.

\bibitem{EscapilInchauspeJerezHanckes2019}
{\sc P.~Escapil-Inchausp{\'e} and C.~Jerez-Hanckes}, {\em Fast {C}alder\'on
  preconditioning for the electric field integral equation}, IEEE Transactions
  on Antennas and Propagation, 67 (2019), pp.~2555--2564,
  \href{http://dx.doi.org/10.1109/TAP.2019.2891129}{doi:\nolinkurl{10.1109/TAP.2019.2891129}}.

\bibitem{EscapilInchauspeJerezHanckes2021}
{\sc P.~Escapil-Inchausp{\'e} and C.~Jerez-Hanckes}, {\em Bi-parametric
  operator preconditioning}, Computers \& Mathematics with Applications, 102
  (2021), pp.~220--232,
  \href{http://dx.doi.org/10.1016/j.camwa.2021.10.012}{doi:\nolinkurl{10.1016/j.camwa.2021.10.012}}.

\bibitem{EscapilInchauspeJerezHanckes2024}
{\sc P.~Escapil-Inchausp{\'e} and C.~Jerez-Hanckes}, {\em Shape uncertainty
  quantification for electromagnetic wave scattering via first-order sparse
  boundary element approximation}, IEEE Transactions on Antennas and
  Propagation, 72 (2024), pp.~6627--6637,
  \href{http://dx.doi.org/10.1109/TAP.2024.3418199}{doi:\nolinkurl{10.1109/TAP.2024.3418199}}.

\bibitem{EscapilInchauspeJerezHanckes2026}
{\sc P.~Escapil-Inchausp{\'e} and C.~Jerez-Hanckes}, {\em Shape holomorphy and
  sparse approximation of the {M}axwell electric field integral operator}.
\newblock Preprint, submitted to Mathematical Models and Methods in Applied
  Sciences, 2026, \href{http://arxiv.org/abs/2609.00466}{arXiv:2609.00466
  [math.NA]}.

\bibitem{FierroJerezHanckes2020}
{\sc I.~Fierro and C.~Jerez-Hanckes}, {\em Fast {C}alder\'on preconditioning
  for {H}elmholtz boundary integral equations}, Journal of Computational
  Physics, 409 (2020), p.~109355,
  \href{http://dx.doi.org/10.1016/j.jcp.2020.109355}{doi:\nolinkurl{10.1016/j.jcp.2020.109355}}.

\bibitem{FierroPiccardoBetcke2023}
{\sc I.~Fierro-Piccardo and T.~Betcke}, {\em An {OSRC} preconditioner for the
  {EFIE}}, IEEE Transactions on Antennas and Propagation, 71 (2023),
  pp.~3408--3417,
  \href{http://dx.doi.org/10.1109/TAP.2023.3236762}{doi:\nolinkurl{10.1109/TAP.2023.3236762}}.

\bibitem{GrahamPemberySpence2021}
{\sc I.~G. Graham, O.~R. Pembery, and E.~A. Spence}, {\em Analysis of a
  {H}elmholtz preconditioning problem motivated by uncertainty quantification},
  Advances in Computational Mathematics, 47 (2021), p.~68,
  \href{http://dx.doi.org/10.1007/s10444-021-09889-0}{doi:\nolinkurl{10.1007/s10444-021-09889-0}}.

\bibitem{Hiptmair2006}
{\sc R.~Hiptmair}, {\em Operator preconditioning}, Computers \& Mathematics
  with Applications, 52 (2006), pp.~699--706,
  \href{http://dx.doi.org/10.1016/j.camwa.2006.10.008}{doi:\nolinkurl{10.1016/j.camwa.2006.10.008}}.

\bibitem{HiptmairJerezHanckes2012}
{\sc R.~Hiptmair and C.~Jerez-Hanckes}, {\em Multiple traces boundary integral
  formulation for {H}elmholtz transmission problems}, Advances in Computational
  Mathematics, 37 (2012), pp.~39--91,
  \href{http://dx.doi.org/10.1007/s10444-011-9194-3}{doi:\nolinkurl{10.1007/s10444-011-9194-3}}.

\bibitem{KleanthousEtAl2022}
{\sc A.~Kleanthous, T.~Betcke, D.~P. Hewett, P.~Escapil-Inchausp{\'e},
  C.~Jerez-Hanckes, and A.~J. Baran}, {\em Accelerated {C}alder\'on
  preconditioning for {M}axwell transmission problems}, Journal of
  Computational Physics, 458 (2022), p.~111099,
  \href{http://dx.doi.org/10.1016/j.jcp.2022.111099}{doi:\nolinkurl{10.1016/j.jcp.2022.111099}}.

\bibitem{ParksEtAl2006}
{\sc M.~L. Parks, E.~de~Sturler, G.~Mackey, D.~D. Johnson, and S.~Maiti}, {\em
  Recycling {K}rylov subspaces for sequences of linear systems}, SIAM Journal
  on Scientific Computing, 28 (2006), pp.~1651--1674.

\bibitem{PowellElman2009}
{\sc C.~E. Powell and H.~C. Elman}, {\em Block-diagonal preconditioning for
  spectral stochastic finite-element systems}, IMA Journal of Numerical
  Analysis, 29 (2009), pp.~350--375,
  \href{http://dx.doi.org/10.1093/imanum/drn014}{doi:\nolinkurl{10.1093/imanum/drn014}}.

\bibitem{Saad1993}
{\sc Y.~Saad}, {\em A flexible inner-outer preconditioned {GMRES} algorithm},
  SIAM Journal on Scientific Computing, 14 (1993), pp.~461--469,
  \href{http://dx.doi.org/10.1137/0914028}{doi:\nolinkurl{10.1137/0914028}}.

\bibitem{SaufferNedelec2011}
{\sc S.~A. Sauter and C.~Schwab}, {\em Boundary Element Methods}, vol.~39 of
  Springer Series in Computational Mathematics, Springer, 2011.

\bibitem{VanHartenScarabosio2026}
{\sc W.~G. van Harten and L.~Scarabosio}, {\em Preconditioning across parameter
  space for the parametric {H}elmholtz equation}, SIAM Journal on Scientific
  Computing,  (2026).
\newblock To appear. Preprint: arXiv:2504.00886 (2025).

\bibitem{Vecchi1999}
{\sc G.~Vecchi}, {\em Loop-star decomposition of basis functions in the
  discretization of the {EFIE}}, IEEE Transactions on Antennas and Propagation,
  47 (1999), pp.~339--346.

\bibitem{WooWangSchuhSanders1993}
{\sc A.~C. Woo, H.~T.~G. Wang, M.~J. Schuh, and M.~L. Sanders}, {\em Benchmark
  radar targets for the validation of computational electromagnetics programs},
  IEEE Antennas and Propagation Magazine, 35 (1993), pp.~84--89,
  \href{http://dx.doi.org/10.1109/74.210840}{doi:\nolinkurl{10.1109/74.210840}}.

\bibitem{YanJinNie2010}
{\sc S.~Yan, J.-M. Jin, and Z.~Nie}, {\em {EFIE} analysis of low-frequency
  problems with loop-star decomposition and {C}alder\'on multiplicative
  preconditioner}, IEEE Transactions on Antennas and Propagation, 58 (2010),
  pp.~857--867.

\end{thebibliography}

\end{document}